\documentclass[11pt]{article}
\usepackage[dvips,letterpaper,margin=1in]{geometry}
\usepackage[utf8]{inputenc}
\usepackage{amsmath}
\usepackage{amssymb}
\usepackage{amstext}
\usepackage{amsthm}
\usepackage[linesnumbered,ruled,vlined]{algorithm2e}
\usepackage{algorithmic}
\usepackage{fullwidth}
\usepackage{float}
\usepackage{colonequals}
\usepackage{hyperref}
\usepackage{color}

\newtheorem{theorem}{Theorem}[section]
\newtheorem{lemma}[theorem]{Lemma}

\newtheorem{corollary}[theorem]{Corollary}
\newtheorem{conjecture}[theorem]{Conjecture}
\newtheorem{definition}[theorem]{Definition}
\newtheorem{remark}[theorem]{Remark}

\newcommand{\RR}{\mathbb{R}}
\newcommand{\QQ}{\mathbb{Q}}
\newcommand{\NN}{\mathbb{N}}

\newcommand{\PP}{\mathbb{P}}
\newcommand{\EE}{\mathbb{E}}

\newcommand{\Ex}{\mathop{\mathbb{E}}}  % puts symbols below

\newcommand{\sI}{\mathcal{I}}

\newcommand{\sN}{\mathcal{N}}

\newcommand{\sS}{\mathcal{S}}

\newcommand{\sX}{\mathcal{X}}

\newcommand{\one}{{\bf 1}}

\newcommand{\la}{\langle}
\newcommand{\ra}{\rangle}

\newcommand{\sign}{\mathrm{sign}}
\newcommand{\supp}{\mathrm{supp}}

\newcommand{\GOE}{\mathsf{GOE}}

\newcommand{\tp}{\textit{\texttt{p}}}
\newcommand{\tq}{\textit{\texttt{q}}}

\newcommand{\eqd}{\stackrel{\mathrm{(d)}}{=}}

\makeatletter
\renewcommand*{\@fnsymbol}[1]{\ensuremath{\ifcase#1\or *\or \ddagger\or
    \mathsection\or \mathparagraph\or \|\or **\or \dagger\dagger
    \or \ddagger\ddagger \else\@ctrerr\fi}}
\makeatother

\title{Subexponential-Time Algorithms for Sparse PCA}

\date{}

\usepackage{authblk}
\author[1]{Yunzi Ding\thanks{Email: \textit{yding@nyu.edu}. Partially supported by NSF grant DMS-1712730.}}
\author[1]{Dmitriy Kunisky\thanks{Email: \textit{kunisky@cims.nyu.edu}. Partially supported by NSF grants DMS-1712730 and DMS-1719545.}}
\author[1]{Alexander S.\ Wein\thanks{Email: \textit{awein@cims.nyu.edu}. Partially supported by NSF grant DMS-1712730 and by the Simons Collaboration on Algorithms and Geometry.}}
\author[2]{Afonso S.\ Bandeira\thanks{Email: \textit{bandeira@math.ethz.ch}. Most of this work was done while ASB was with the Department of Mathematics at the Courant Institute of Mathematical Sciences, and the Center for Data Science, at New York University; and partially supported by NSF grants DMS-1712730 and DMS-1719545, and by a grant from the Sloan Foundation.}}

\affil[1]{Department of Mathematics, Courant Institute of Mathematical Sciences, New York University, USA}
\affil[2]{Department of Mathematics, ETH Zurich, Switzerland}

\begin{document}

\maketitle
\thispagestyle{empty}

\begin{abstract}
We study the computational cost of recovering a unit-norm sparse principal component $x \in \RR^n$ planted in a random matrix, in either the Wigner or Wishart spiked model (observing either $W + \lambda xx^\top$ with $W$ drawn from the Gaussian orthogonal ensemble, or $N$ independent samples from $\sN(0, I_n + \beta xx^\top)$, respectively).
Prior work has shown that when the signal-to-noise ratio ($\lambda$ or $\beta\sqrt{N/n}$, respectively) is a small constant and the fraction of nonzero entries in the planted vector is $\|x\|_0 / n = \rho$, it is possible to recover $x$ in polynomial time if $\rho \lesssim 1/\sqrt{n}$. While it is possible to recover $x$ in exponential time under the weaker condition $\rho \ll 1$, it is believed that polynomial-time recovery is impossible unless $\rho \lesssim 1/\sqrt{n}$. We investigate the precise amount of time required for recovery in the ``possible but hard'' regime $1/\sqrt{n} \ll \rho \ll 1$ by exploring the power of subexponential-time algorithms, i.e., algorithms running in time $\exp(n^\delta)$ for some constant $\delta \in (0,1)$. For any $1/\sqrt{n} \ll \rho \ll 1$, we give a recovery algorithm with runtime roughly $\exp(\rho^2 n)$, demonstrating a smooth tradeoff between sparsity and runtime.
Our family of algorithms interpolates smoothly between two existing algorithms: the polynomial-time diagonal thresholding algorithm and the $\exp(\rho n)$-time
exhaustive search algorithm.
Furthermore, by analyzing the low-degree likelihood ratio, we give rigorous evidence suggesting that the tradeoff achieved by our algorithms is optimal.
\end{abstract}

\newpage
\tableofcontents

\newpage

\section{Introduction}

\subsection{Spiked Matrix Models}

Since the foundational work of Johnstone \cite{J-spiked}, spiked random matrix ensembles have been widely studied throughout random matrix theory, statistics, and theoretical data science. These models describe a deformation of one of several canonical random matrix distributions by a rank-one perturbation or ``spike,'' intended to capture a signal corrupted by noise. Spectral properties of these spiked models have received much attention in random matrix theory \cite{BBP,BS-wishart,paul-wishart,peche,FP,CDF,BGN,PRS,KY}, leading to a theoretical understanding of methods based on principal component analysis (PCA) for recovering the direction of the rank-one spike \cite{J-spiked,JL04,paul,nadler,JL09}. Spiked matrix models have also found more specific applications to problems such as community detection in graphs (see, e.g., \cite{mcsherry,vu-svd,DAM}, or  \cite{moore-survey,abbe-survey} for surveys) and synchronization over groups (see, e.g., \cite{singer-angular,singer-cryo,sdp-phase,PWBM-contig-synch,PWBM-amp}).

We will study two classical variants of the spiked matrix model: the \emph{Wigner} and \emph{Wishart} models.
The models differ in how noise is applied to the signal vector.
In either case, let $x \in \RR^n$ be the signal vector (or ``spike'').
We will either have $x$ deterministic with $\|x\|_2 = 1$, or $x \in \RR^n$ random for each $n$ with $\|x\|_2 \to 1$ in probability as $n \to \infty$.

\begin{itemize}
\item \textbf{Spiked Wigner Model.} 
Let $\lambda > 0$. Observe $Y = W + \lambda xx^\top$, where $W \in \RR^{n \times n}$ is drawn from the \emph{Gaussian orthogonal ensemble} $\GOE(n)$, i.e., $W$ is symmetric with entries distributed independently as $W_{ii} \sim \sN(0,2/n)$ for all $1 \leq i \leq n$, and $W_{ij} = W_{ji} \sim \sN(0,1/n)$ for all $1 \leq i < j \leq n$.

\item \textbf{Spiked Wishart Model.} 
Let $\beta > 0$ and $N \in \NN$.
Observe $N$ samples $y^{(1)},y^{(2)},\dots,y^{(N)} \in \RR^n$ drawn independently from $\sN(0,I_n + \beta xx^\top)$.
The ratio of dimension to number of samples is denoted $\gamma \colonequals n/N$. We will focus on the \emph{high-dimensional} regime where $\gamma$ converges to a constant as $n \to \infty$. We let $Y$ denote the \emph{sample covariance matrix} $Y = \frac{1}{N} \sum_{i=1}^N y^{(i)} {y^{(i)}}^\top$.
\end{itemize}

\noindent
Each of these \emph{planted models} has a corresponding \emph{null model}, given by sampling from the planted model with either $\lambda = 0$ (Wigner) or $\beta = 0$ (Wishart). 

We are interested in the computational feasibility of the following two statistical tasks, to be performed given a realization of the \emph{data} (either $Y$ or $\{y^{(1)},\ldots,y^{(N)}\}$) drawn from either the null or planted distribution.
\begin{itemize}
\item \textbf{Detection.} Perform a simple hypothesis test between the planted model and null model. We say that \emph{strong detection} is achieved by a statistical test if both the type-I and type-II errors tend to 0 as $n \to \infty$.
\item \textbf{Recovery.} Estimate the spike $x$ given data drawn from the planted model. We say that a unit-norm estimator $\widehat x \in \RR^n$ achieves \emph{weak recovery} if $\langle \widehat x,x \rangle^2$ remains bounded away from zero with probability tending to $1$ as $n \to \infty$.\footnote{We will also consider stronger notions of recovery: \emph{strong recovery} is $\langle \hat x, x \rangle^2 \to 1$ as $n \to \infty$ and \emph{exact recovery} is $\hat x = x$ with probability $1-o(1)$.} (Note that we cannot hope to distinguish between the planted models with signals $x$ and $-x$.)
\end{itemize}
For high-dimensional inference problems such as the spiked Wigner and Wishart models, these two tasks typically share the same computational profile: with a given computational time budget, strong detection and weak recovery are possible in the same regions of parameter space.

\subsection{Principal Component Analysis}

Simple algorithms for both detection and recovery are given by \emph{principal component analysis (PCA)} of the matrix $Y$.
For detection, one computes and thresholds the maximum eigenvalue $\lambda_{\max}(Y)$ of $Y$, while for recovery one estimates $x$ using the leading eigenvector $v_{\max}(Y)$.
Both the spiked Wishart and Wigner models are known to exhibit a sharp transition in their top eigenvalue as the model parameters vary. For the Wishart model, the celebrated ``BBP transition'' of Baik, Ben Arous, and P\'ech\'e \cite{BBP,BS-wishart} states that the maximum eigenvalue of the sample covariance matrix $Y$ emerges from the Marchenko--Pastur-distributed bulk if and only if $\beta^2 > \gamma$. Similarly, in the Wigner model, the maximum eigenvalue of $Y$ emerges from the semicircular bulk if and only if $\lambda > 1$ \cite{FP}. More formally, the following statements hold.

\begin{theorem}[\cite{FP,BGN}]
\label{thm:wigner-bbp}
Consider the spiked Wigner model $Y = W+\lambda xx^\top$ with $\|x\| = 1$ and $\lambda > 0$ fixed. Then as $n\rightarrow\infty$,
\begin{itemize}
\item if $\lambda \le 1$, $\lambda_{\max}(Y) \rightarrow 2$ almost surely, and $\la v_{\max}(Y),x\ra^2 \rightarrow 0$ almost surely (where $v_{\max}$ denotes the leading eigenvector);
\item if $\lambda > 1$, $\lambda_{\max}(Y) \rightarrow \lambda + \lambda^{-1} > 2$ almost surely, and $\la v_{\max}(Y),x \ra^2 \rightarrow 1-\lambda^{-2} > 0$ almost surely.
\end{itemize}
\end{theorem}

\begin{theorem}[\cite{BBP,BS-wishart,paul-wishart}]
\label{thm:wishart-bbp}
Let $Y$ denote the sample covariance matrix in the spiked Wishart model with $\|x\| = 1$, $\beta > 0$ fixed, and $N = N(n)$ such that $\gamma \colonequals n/N$ converges to a constant $\bar\gamma > 0$ as $n \to \infty$. Then as $n \to \infty$,
\begin{itemize}
\item if $\beta^2 \le \bar\gamma$, $\lambda_{\max}(Y) \rightarrow (1+\sqrt{\bar\gamma})^2$ almost surely, and $\la v_{\max}(Y),x\ra^2 \rightarrow 0$ almost surely (where $v_{\max}$ denotes the leading eigenvector);
\item if $\beta^2 > \bar\gamma$, $\lambda_{\max}(Y) \rightarrow (1+\beta)(1+\bar\gamma/\beta) > (1+\sqrt{\bar\gamma})^2$ almost surely, and $\la v_{\max}(Y),x \ra^2 \rightarrow (1-\bar\gamma/\beta^2)(1+\bar\gamma/\beta) > 0$ almost surely.
\end{itemize}
\end{theorem}

\noindent We define the \emph{signal-to-noise ratio (SNR)}  $\hat\lambda$ by
\begin{equation}
    \hat\lambda \colonequals \left\{\begin{array}{ll} \lambda & \text{in the Wigner model,} \\ \beta / \sqrt{\gamma} & \text{in the Wishart model.} \end{array}\right.
\end{equation}
Theorems~\ref{thm:wigner-bbp} and \ref{thm:wishart-bbp} then characterize the performance of PCA in terms of $\hat\lambda$. Namely, thresholding the largest eigenvalue of $Y$ succeeds at strong detection when $\hat\lambda > 1$ and fails when $\hat\lambda \le 1$; similarly, the top eigenvector succeeds at weak recovery when $\hat\lambda > 1$ and fails when $\hat\lambda \le 1$. For some distributions of $x$, including the spherical prior ($x$ drawn uniformly from the unit sphere) and the Rademacher prior (each entry $x_i$ drawn i.i.d.\ from $\mathrm{Unif}(\pm 1/\sqrt{n})$), it is known that the PCA threshold is optimal, in the sense that strong detection and weak recovery are statistically impossible (for \emph{any} test or estimator, regardless of computational cost) when $\hat\lambda < 1$ \cite{sphericity,MRZ,DAM,BMVVX,PWBM-contig}.

\subsection{Sparse PCA}

\emph{Sparse PCA}, a direction initiated by Johnstone and Lu \cite{JL04,JL09}, seeks to improve performance when the planted vector is known to be sparse in a given basis. While various sparsity assumptions have been considered, a simple and illustrative one is to take $x$ drawn from the \emph{sparse Rademacher prior}, denoted $\mathcal{X}_n^\rho$, in which each entry $x_i$ is distributed independently as
\begin{equation}
x_i = \left\{
\begin{aligned}
\frac{1}{\sqrt{\rho n}} \quad &\text{with probability }\frac{\rho}{2}\\
-\frac{1}{\sqrt{\rho n}} \quad &\text{with probability }\frac{\rho}{2}\\
0 \quad &\text{with probability }1-\rho\\
\end{aligned}
\right.
\end{equation}
for some known sparsity parameter $\rho \in (0,1]$, which may depend on $n$.\footnote{We analyze our algorithms for a more general set of assumptions on $x$; see Definition~\ref{gen-sps-rad}.} The normalization ensures $\|x\| \to 1$ in probability as $n \to \infty$.

Consider the Wishart model (we will see that the Wigner model shares essentially the same behavior) in the regime $\hat\lambda = \Theta(1)$ with $\hat\lambda < 1$ (so that ordinary PCA fails at weak recovery). The simple \emph{diagonal thresholding} algorithm proposed by \cite{JL09} estimates the support of $x$ by identifying the largest diagonal entries of $Y$. Under the condition\footnote{We use $A \lesssim B$ to denote $A \le CB$ for some constant $C$, and use $A \ll B$ to denote $A \le B/\mathrm{polylog}(n)$.} $\rho \lesssim 1/\sqrt{n \log n}$, this has been shown \cite{AW-sparse} to achieve \emph{exact support recovery}, i.e., it exactly recovers the support of $x$ with probability tending to $1$ as $n \to \infty$ (and once the support is known, it is straightforward to recover $x$). The more sophisticated \emph{covariance thresholding} algorithm proposed by \cite{KNV} has been shown \cite{DM-sparse} to achieve exact support recovery when $\rho \lesssim 1/\sqrt{n}$.

On the other hand, given unlimited computational power, an exhaustive search over all possible support sets of size $\rho n$ achieves exact support recovery under the much weaker assumption $\rho \lesssim 1/\log(n)$ \cite{PJ-aug,VL-minimax,CMW}. Similarly, strong detection and weak recovery are statistically possible even when $\rho$ is a sufficiently small constant (depending on $\beta,\gamma$) \cite{BMVVX,PWBM-contig}, and the precise critical constant $\rho^*(\beta,\gamma)$ is given by the \emph{replica formula} from statistical physics (see, e.g., \cite{LKZ-mmse,LKZ-sparse,KXZ-mi,replica-proof,LM-symm,leo-asymm,finite-corr,replica-short,detection-wig}, or \cite{leo-survey} for a survey). However, no polynomial-time algorithm is known to succeed (for any reasonable notion of success) when $\rho \gg 1/\sqrt{n}$, despite extensive work on algorithms for sparse PCA \cite{dGJL,ZHT,MWA,dBG,AW-sparse,WTH,BR-opt, opt-sparse-pca, KNV,sparse-from-lin}. 

In fact, a growing body of theoretical evidence suggests that no polynomial-time algorithm can succeed when $\rho \gg 1/\sqrt{n}$ \cite{BR-sparse,CMW,MW-sos,WBS,sos-hidden,bresler-sparse,bresler-pca}. Such evidence takes the form of reductions \cite{BR-sparse,WBS,bresler-sparse,bresler-pca} from the \emph{planted clique} problem (which is widely conjectured to be hard in certain regimes \cite{J-clique,DM-clique,MPW-clique,pcal}), as well as lower bounds against the sum-of-squares hierarchy of convex relaxations \cite{MW-sos,sos-hidden}. Thus, we expect sparse PCA to exhibit a large ``possible but hard'' regime when $1/\sqrt{n} \ll \rho \ll 1$. Statistical-to-computational gaps of this kind are believed to occur and have been studied extensively in many other statistical inference problems, such as community detection in the stochastic block model \cite{block-model-1,block-model-2} and tensor PCA \cite{RM-tensor,HSS-tensor,sos-hidden,ZX-tensor,secret-leakage}.

\subsection{Our Contributions}

In this paper, we investigate precisely how hard the ``hard'' region ($1/\sqrt{n} \ll \rho \ll 1$) is in sparse PCA. We consider subexponential-time algorithms, i.e., algorithms with runtime $\exp(n^{\delta + o(1)})$ for fixed $\delta \in (0,1)$. We show a smooth tradeoff between sparsity (governed by $\rho$) and runtime (governed by $\delta$). More specifically, our results (for both the Wishart and Wigner models) are as follows.
\begin{itemize}
\item \textbf{Algorithms.} For any $\delta \in (0,1)$, we give an algorithm with runtime $\exp(n^{\delta + o(1)})$ that achieves exact support recovery, provided $\rho \ll n^{(\delta-1)/2}$.
\item \textbf{Lower bounds.} Through an analysis of the \emph{low-degree likelihood ratio} (see Section~\ref{sec:low-degree}), we give formal evidence suggesting that the above condition is essentially tight in the sense that no algorithm of runtime $\exp(n^{\delta + o(1)})$ can succeed when $\rho \gg n^{(\delta-1)/2}$. (Our results are sharper than the sum-of-squares lower bounds of \cite{sos-hidden} in that we pin down the precise constant $\delta$.)
\end{itemize}
Our algorithm involves exhaustive search over subsets of $[n]$ of cardinality $\ell \approx n^\delta$. The case $\ell = 1$ is diagonal thresholding (which is polynomial-time and succeeds when $\rho \lesssim 1/\sqrt{n \log n}$) and the case $\ell = \rho n$ is exhaustive search over all possible spikes (which requires time $\exp(\rho n^{1+o(1)})$ and succeeds when $\rho \lesssim 1/(\log n)$). As $\ell$ varies in the range $1 \le \ell \le \rho n$, our algorithm interpolates smoothly between these two extremes. For a given $\rho$ in the range $1/\sqrt{n} \ll \rho \ll 1$, the smallest admissible choice of $\ell$ is roughly $\rho^2 n$, yielding an algorithm of runtime $\exp(\rho^2 n^{1+o(1)})$. 

Our results extend to the case $\hat\lambda \ll 1$, e.g., $\hat\lambda = n^{-\alpha}$ for some constant $\alpha > 0$. In this case, provided $\rho \ll \hat\lambda^2$ (which is information-theoretically necessary \cite{PJ-aug,VL-minimax,CMW}), there is an $\exp(n^{\delta + o(1)})$-time algorithm if $\rho \ll \hat\lambda n^{(\delta-1)/2}$, and the low-degree likelihood ratio again suggests that this is optimal. In other words, for a given $\rho$ in the range $\hat\lambda/\sqrt{n} \ll \rho \ll \hat\lambda^2$, we can solve sparse PCA in time $\exp(\hat\lambda^{-2} \rho^2 n^{1+o(1)})$.

The analysis of our algorithm applies not just to the sparse Rademacher spike prior, but also to a weaker set of assumptions on the spike that do not require all of the nonzero entries to have the same magnitude. Our algorithm is guaranteed (with high probability) to exactly recover both the support of $x$ and the signs of the nonzero entries of $x$. Once the support is known, it is straightforward to estimate $x$ via the leading eigenvector of the appropriate submatrix.

In independent work \cite{anytime-pca}, a different algorithm for sparse PCA was proposed and shown to have essentially the same subexponential runtime as ours. Also, prior work \cite{rip-subexp} gave a subexponential-time algorithm for certifying the restricted isometry property, that is somewhat similar in spirit to our algorithm for sparse PCA.

\begin{remark}
Certain problems besides sparse PCA have a similar smooth tradeoff between subexponential runtime requirements and statistical power. These include refuting random constraint satisfaction problems \cite{strong-refute} and tensor PCA \cite{mult-approx,cert-tensor,kikuchi-tensor}. In contrast, other problems have a sharp threshold at which they transition from being solvable in polynomial-time to (conjecturally) requiring essentially exponential time: $\exp(n^{1-o(1)})$. Examples of this behavior can occur at the spectral transition at $\hat\lambda = 1$ in the spiked Wishart and Wigner matrix models (see \cite{BKW-sk,low-deg-notes}) as well as at the Kesten--Stigum threshold in the stochastic block model (see \cite{block-model-1,block-model-2,HS,sam-thesis}).
\end{remark}

\subsection{Background on the Low-Degree Likelihood Ratio}
\label{sec:low-degree}

A sequence of recent work on the sum-of-squares hierarchy \cite{pcal,HS,sos-hidden,sam-thesis} has led to the development of a remarkably simple method for predicting the amount of computation time required to solve statistical tasks. This method---which we will refer to as the \emph{low-degree method}---is based on analyzing the so-called \emph{low-degree likelihood ratio}, and is believed to be intimately connected to the power of sum-of-squares (although formal implications have not been established). We now give an overview of this method; see \cite{sam-thesis,low-deg-notes} for more details.

We will consider the problem of distinguishing two simple hypotheses $\PP_n$ and $\QQ_n$, which are probability distributions on some domain $\Omega_n = \RR^{d(n)}$ with $d(n) = \mathrm{poly}(n)$. The idea of the low-degree method is to explore whether there is a low-degree polynomial $f_n: \Omega_n \to \RR$ that can distinguish $\PP_n$ from $\QQ_n$.

We call $\QQ_n$ the ``null'' distribution, which for us will always be i.i.d.\ Gaussian (see Definitions~\ref{def:wishart} and \ref{def:wigner}).
$\QQ_n$ induces an inner product on $L^2$ functions $f: \Omega_n \to \RR$ given by $\la f,g \ra_{L^2(\QQ_{n})} = \EE_{Y \sim \QQ_n}[f(Y)g(Y)]$, and a norm $\|f\|_{L^2(\QQ_n)}^2 = \la f,f \ra_{L^2(\QQ_n)}$. For $D \in \NN$, let $\RR[Y]_{\le D}$ denote the multivariate polynomials $\Omega_n \to \RR$ of degree at most $D$. For $f: \Omega_n \to \RR$, let $f^{\le D}$ denote the orthogonal projection (with respect to $\la \cdot,\cdot \ra_{L^2(\QQ_n)}$) of $f$ onto $\RR[Y]_{\le D}$. The following result then relates the distinguishing power of low-degree polynomials (in a certain $L^2$ sense) to the \emph{low-degree likelihood ratio}.

\begin{theorem}[\cite{HS,sos-hidden}]
Let $\PP$ and $\QQ$ be probability distributions on $\Omega = \RR^d$. Suppose $\PP$ is absolutely continuous with respect to $\QQ$, so that the likelihood ratio $L = \frac{d\PP}{d\QQ}$ is defined. Then
\begin{equation}\label{eq:max-f}
\max_{f \in \RR[Y]_{\le D} \setminus \{0\}} \frac{\EE_{Y \sim \PP}[f(Y)]}{\sqrt{\EE_{Y \sim \QQ}[f(Y)^2]}} = \|L^{\le D}\|_{L^2(\QQ)}.
\end{equation}
\end{theorem}
\noindent
(The proof is straightforward: the fraction on the left can be written as $\langle f,L \rangle_{L^2(\QQ_n)} / \|f\|_{L^2(\QQ_n)}$, so the maximizer is $f = L^{\le D}$.) 
The left-hand side of~\eqref{eq:max-f} is a heuristic measure of how well degree-$D$ polynomials can distinguish $\PP$ from $\QQ$: if this quantity is $O(1)$ as $n \to \infty$, this suggests that no degree-$D$ polynomial can achieve strong detection (and indeed this is made formal by Theorem~4.3 of \cite{low-deg-notes}).
The right-hand side of~\eqref{eq:max-f} is the norm of the low-degree likelihood ratio (LDLR), which can be computed or bounded in many cases, making this heuristic a practical tool for predicting computational feasibility of hypothesis testing.

The key assumption underlying the low-degree method is that, for many natural distributions $\PP_n$ and $\QQ_n$, degree-$D$ polynomials are as powerful as algorithms of runtime $n^{\tilde\Theta(D)}$, where $\tilde\Theta$ hides factors of $\log n$. This is captured by the following informal conjecture, which is based on \cite{HS,sos-hidden,sam-thesis}; in particular, see Hypothesis~2.1.5 of \cite{sam-thesis}.

\begin{conjecture}[Informal]
    \label{conj:low-deg-informal}
    Suppose $t: \NN \to \NN$.
    For ``nice'' sequences of distributions $\PP_n$ and $\QQ_n$, if $\|L_n^{\leq D(n)}\|_{L^2(\QQ_n)}$ remains bounded as $n \to \infty$ whenever $D(n) \le t(n) \cdot \mathrm{polylog}(n)$, then there exists no sequence of functions $f_n: \Omega_n \to \{\tp,\tq\}$ with $f_n$ computable in time $n^{t(n)}$ that strongly distinguishes $\PP_n$ and $\QQ_n$, i.e., that satisfies
    \begin{equation}
        \lim_{n \to \infty} \QQ_n\left[ f_n(Y) = \tq \right] = \lim_{n \to \infty} \PP_n\left[ f_n(Y) = \tp \right] = 1.
    \end{equation}
\end{conjecture}
\noindent
On a finer scale, it is conjectured \cite{HS,sam-thesis} that if for some $\varepsilon > 0$ we have $D(n) \ge \log^{1+\varepsilon}(n)$ and $\|L_n^{\leq D(n)}\|_{L^2(\QQ_n)} = O(1)$, then no polynomial-time algorithm can strongly distinguish $\PP_n$ from $\QQ_n$. In practice, it seems that the converse of Conjecture~\ref{conj:low-deg-informal} often holds as well, in the sense that if $\|L_n^{\le D(n)}\|_{L^2(\QQ_n)} = \omega(1)$ for some $D(n) = t(n) / \mathrm{polylog}(n)$, then there is an $n^{t(n)}$-time distinguishing algorithm (however, see Remark~\ref{rem:low-deg-fail} for one caveat).

Calculations with the LDLR have been carried out for problems such as community detection \cite{HS,sam-thesis}, planted clique \cite{pcal,sam-thesis}, the spiked Wishart model \cite{BKW-sk}, the spiked Wigner model \cite{low-deg-notes}, and tensor PCA \cite{sos-hidden,sam-thesis,low-deg-notes} (tensor PCA exhibits a subexponential-time tradeoff similar to sparse PCA; see \cite{low-deg-notes}). In all of the above cases, the low-degree predictions coincide with widely-conjectured statistical-versus-computational tradeoffs.

Various leading algorithmic approaches can be approximated by low-degree polynomials and are thus ruled out by low-degree lower bounds of the form $\|L_n^{\leq D(n)}\|_{L^2(\QQ_n)} = O(1)$. These approaches include a general class of spectral methods (see Theorem~4.4 of \cite{low-deg-notes}) as well as the algorithms that we present in this paper (see Remark~\ref{rem:alg-lowdeg}). The low-degree predictions are also conjectured to coincide with the power of the sum-of-squares hierarchy and are in particular connected to the \emph{pseudo-calibration} approach \cite{pcal}; see \cite{sos-hidden,sos-survey,sam-thesis}. We refer the reader to Section~4 of \cite{low-deg-notes} for further discussion of the implications (both formal and conjectural) of low-degree lower bounds.

Conjecture \ref{conj:low-deg-informal} is informal in the sense that we have not specified the meaning of ``nice" $\PP_n$ and $\QQ_n$. Roughly speaking, highly-symmetric high-dimensional problems are considered ``nice'' so long as $\PP_n$ and $\QQ_n$ have at least a small amount of noise in order to rule out brittle high-degree algorithms such as Gaussian elimination. (In particular, we consider spiked Wigner and Wishart to be ``nice.'') Conjecture~2.2.4 of \cite{sam-thesis} is one formal variant of the low-degree conjecture, although it uses the more refined notion of \emph{coordinate degree} and so does not apply to the calculations in this paper.

We remark that if $\|L_n\|_{L^2(\QQ_n)} = O(1)$ (the $D = \infty$) case, then it is statistically impossible to strongly distinguish $\PP_n$ and $\QQ_n$; this is a commonly-used \emph{second moment method} (see e.g., \cite{MRZ,BMVVX,PWBM-contig}) of which Conjecture~\ref{conj:low-deg-informal} is a computationally-bounded analogue.

In this paper we give tight computational lower bounds for sparse PCA, conditional on Conjecture~\ref{conj:low-deg-informal}. Alternatively, one can view the results of this paper as a ``stress test'' for Conjecture~\ref{conj:low-deg-informal}: we show that Conjecture~\ref{conj:low-deg-informal} predicts a certain statistical-versus-computational tradeoff and this indeed matches the best algorithms that we know.

\paragraph{Organization.}
The remainder of the paper is organized as follows.
In Section~\ref{sec:main-results}, we present our subexponential-time algorithms and our lower bounds based on the low-degree likelihood ratio.
In Section~\ref{sec:proofs-alg}, we give proofs for the correctness of our algorithms. In Section~\ref{sec:proofs-ldlr}, we give proofs for our analysis of the low-degree likelihood ratio.

\paragraph{Notation.}
We use standard asymptotic notation $O(\cdot)$, $\Omega(\cdot)$, $\Theta(\cdot)$, always pertaining to the limit $n \to \infty$. We also use $\tilde{O}(B)$ to mean $O(B \cdot \mathrm{polylog}(n))$ and $\tilde{\Omega}(B)$ to mean $\Omega(B / \mathrm{polylog}(n))$. Also recall that $f(n) = o(g(n))$ means $f(n)/g(n) \to 0$ as $n \to \infty$ and $f(n) = \omega(g(n))$ means $f(n)/g(n) \to \infty$ as $n \to \infty$. An event occurs with \emph{high probability} if it occurs with probability $1-o(1)$. We sometimes use the shorthand $A \lesssim B$ to mean $A \le CB$ for an absolute constant $C$, and the shorthand $A \ll B$ to mean $A \le B/\mathrm{polylog}(n)$.

\section{Main Results}
\label{sec:main-results}

In the analysis of our algorithms, we consider the spiked Wishart and Wigner models with signal $x$ satisfying the following properties.
\begin{definition}\label{gen-sps-rad}
For $\rho \in (0,1]$ and $A \ge 1$, a vector $x\in\mathbb{R}^n$ is called \emph{$(\rho,A)$-sparse} if
\begin{itemize}
\item $\|x\|_2 = 1$ and $\|x\|_0 = \rho n$, and
\item for any $i\in{\rm supp}(x)$, $\frac{1}{A\sqrt{\rho n}} \le |x_i| \le \frac{A}{\sqrt{\rho n}}$.
\end{itemize}
\end{definition}
\noindent Here we have used the standard notations ${\rm supp}(x) = \{i \in [n] \,:\, x_i \ne 0\}$ and $\|x\|_0 = |{\rm supp}(x)|$. We assume that $\rho$ (which may depend on $n$) is chosen so that $\rho n$ is an integer.

\begin{remark}
A lower bound on $|x_i|$ is essential for exact support recovery, since we cannot hope to distinguish tiny nonzero entries of $x$ from zero entries. The upper bound on $|x_i|$ is a technical condition that is likely not essential, and is only used for recovery in the Wishart model (Theorem~\ref{rec_wsh}).
\end{remark}

In our calculations of the low-degree likelihood ratio, we instead assume the signal $x$ is drawn from the sparse Rademacher distribution, defined as follows.
\begin{definition}\label{sps-rad}
The \emph{sparse Rademacher} prior $\sX_n^\rho$ with sparsity $\rho \in (0,1]$ is the distribution on $\mathbb{R}^n$ whereby $x\sim\sX_n^\rho$ has i.i.d.\ entries distributed as
\begin{equation}\label{sparse-rad}
x_i = \left\{
\begin{array}{cll}
+1 / \sqrt{\rho n} & \text{ with probability } & \rho / 2, \\
-1 / \sqrt{\rho n} & \text{ with probability } & \rho / 2, \\
0 & \text{ with probability } & 1 - \rho.
\end{array}
\right.
\end{equation}
Note that $x\sim \sX_n^\rho$ has $\|x\|_2 \to 1$ in probability as $n \to \infty$.
\end{definition}

\subsection{The Wishart Model}

We first present our results for the Wishart model. Our algorithms and results for the Wigner model are essentially identical and can be found in Section~\ref{sec:wigner}.

\begin{definition}[Spiked Wishart model]
\label{def:wishart}
The spiked Wishart model with parameters $n, N \in \NN_+$, $\beta \ge 0$, and planted signal $x \in \RR^n$ is defined as follows.
\begin{itemize}
\item Under $\PP_n = \PP_{n,N,\beta}$, we observe $N$ independent samples $y^{(1)},\ldots,y^{(N)} \sim \mathcal{N}(0,I_n + \beta x x^\top)$.
\item Under $\QQ_n = \QQ_{n,N}$, we observe $N$ independent samples $y^{(1)}, \dots, y^{(N)} \sim \mathcal{N}(0,I_n)$.
\end{itemize}
We will sometimes specify a prior $\mathcal{X}_n$ for $x$, in which case $\PP_n$ first draws $x \sim \mathcal{X}_n$ and then draws $y^{(1)},\ldots,y^{(N)}$ as above.
\end{definition}

\paragraph{Detection.} We first consider the detection problem, where the goal is to determine whether the given data $\{y^{(i)}\}$ was drawn from $\PP_n$ or $\QQ_n$.

\vspace{5pt}
\noindent\begin{minipage}{\linewidth}
\begin{algorithm}[H]
\caption{Detection in the spiked Wishart model}
\label{wishart-detection}
\begin{algorithmic}[1]
\REQUIRE Data $\{y^{(i)}\}_{1\le i\le N}$, parameters $\rho \in (0,1]$, $\beta \ge 0$, $A \ge 1$, $\ell \in \NN_+$
\STATE Compute the sample covariance matrix: $Y \leftarrow \frac{1}{N}\sum_{i = 1}^N y^{(i)}{y^{(i)}}^\top$
\STATE Specify the search set: $\sI_{n,\ell} \leftarrow \{v\in \{-1,0,1\}^n \ :\ \|v\|_0 = \ell\}$
\STATE Compute the test statistic: $T\leftarrow \max_{v \in \sI_{n,\ell}} v^{\top}Yv$
\STATE Compute the threshold: $T^* \leftarrow \ell(1+\frac{\beta \ell}{2A^2\rho n})$
\IF{$T \ge T^*$}
\RETURN $\tp$
\ELSE
\RETURN $\tq$
\ENDIF
\end{algorithmic}
\end{algorithm}
\end{minipage}
\vspace{5pt}

\noindent
The detection algorithm is motivated by the fact that $v^\top Yv = \frac{1}{N}\sum_{i = 1}^N \la v,y^{(i)} \ra^2$. Under the planted model $\PP_n$, $y^{(i)} \sim \sN (0,I_n+\beta xx^\top)$ and thus $\la v,y^{(i)}\ra \sim \sN(0, \ell+\beta \la v,x\ra^2)$ for any fixed $v \in \sI_{n,\ell}$; as a result, if $v$ correctly ``guesses" $\ell$ entries of $x$ with correct signs (up to a global flip), then the contribution of $\la v,x\ra^2$ to the variance of $\la v,y^{(i)}\ra$ will cause $v^\top Y v$ to be large.

\begin{remark}[Runtime]
\label{rem:runtime}
The runtime of Algorithm~\ref{wishart-detection} is dominated by exhaustive search over $\sI_{n,\ell}$ during Step 3, when we compute $T$. Since $|\sI_{n,\ell}| = \binom{n}{\ell} 2^{\ell} \le (2n)^\ell$, the runtime is $n^{O(\ell)}$. If $\ell = \lceil n^\delta \rceil$ for a constant $\delta > 0$, then the runtime is $n^{O(n^\delta)} = \exp(n^{\delta + o(1)})$.
\end{remark}

\begin{theorem}[Wishart detection]\label{det_wsh}
Consider the spiked Wishart model with a $(\rho,A)$-sparse signal $x$, and let $\gamma = n/N$. Let $\{y^{(i)}\}_{i = 1}^N$ be drawn from either $\PP_n$ or $\QQ_n$, and let $f_n$ be the output of Algorithm~\ref{wishart-detection}. Suppose
\begin{equation}\label{eq:wsh-det-rho}
\rho \le \min\left(1,\frac{\beta}{A^2}\right)\frac{\beta}{25A^2\gamma} \frac{1}{\log n}.
\end{equation}
Let $\ell$ be any integer in the interval
\begin{equation}\label{eq:wsh-det-ell}
\ell \in \left[\frac{25 A^4\gamma}{\beta^2}\rho^2 n\log n,\;\min\left(1,\frac{\beta}{A^2}\right)\frac{A^2}{\beta}\rho n\right],
\end{equation}
which is nonempty due to \eqref{eq:wsh-det-rho}.
Then, the total failure probability of Algorithm~\ref{wishart-detection}  satisfies
$$\PP_n[f_n = \tq] + \QQ_n[f_n = \tp] \le 2\exp\left(-\frac{\beta^2}{48A^4\gamma}\frac{\ell^2}{\rho^2 n}\right) \le 2n^{-25\ell/48},$$
where the last inequality follows from \eqref{eq:wsh-det-ell}.
\end{theorem}

\begin{remark}\label{rem:wsh-best-ell}
Since the runtime is $n^{O(\ell)}$, for the best possible runtime we should choose $\ell$ as small as possible, i.e.,
$$\ell = \left\lceil\frac{25A^4\gamma}{\beta^2}\rho^2 n\log n \right\rceil.$$
\end{remark}

\begin{remark}
We are primarily interested in the regime $n \to \infty$ with $\gamma = \Theta(1)$, $A = \Theta(1)$, $\rho = n^{-\tau}$ for a constant $\tau \in (0,1)$, and either $\beta = n^{-\alpha}$ for a constant $\alpha > 0$, or $\beta = \Theta(1)$ with $\hat\lambda \colonequals \beta/\sqrt{\gamma} < 1$ (in which case $\alpha \colonequals 0$). In this case, the requirement \eqref{eq:wsh-det-rho} reads $\rho \le \Omega(\hat\lambda^2/\log n)$ (or, in other words, $\tau > 2 \alpha$), which is information-theoretically necessary up to log factors \cite{PJ-aug,VL-minimax,CMW}. Choosing $\ell$ as in Remark~\ref{rem:wsh-best-ell} yields an algorithm of runtime $n^{O(1 + \hat\lambda^{-2} \rho^2 n \log n)} = \mathrm{poly}(n) + \exp(n^{2\alpha - 2\tau + 1 + o(1)})$.
\end{remark}

\begin{remark}
For $\sS \subseteq [n]$, let $Y_\sS$ denote the corresponding principal submatrix of $Y$ (i.e., restrict to the rows and columns whose indices lie in $\sS$). An alternative detection algorithm would be to threshold the test statistic
$$T' \colonequals \max_{\sS \in \binom{[n]}{\ell}} \lambda_{\max}(Y_\sS),$$
i.e., the largest eigenvalue of any $\ell \times \ell$ principal submatrix. One can obtain similar guarantees for this algorithm as for Algorithm~\ref{wishart-detection}.
\end{remark}

\paragraph{Recovery.} We now turn to the problem of exactly recovering the support and signs of $x$, given data drawn from $\PP_n$. The goal is to output a vector $\bar{x} \in \{-1,0,1\}^n$ such that $\sign(\bar x) = \pm \sign(x)$ where $\sign(x)_i = \sign(x_i)$ and
$$\sign(x_i) = \left\{\begin{array}{rl} 1 & \text{if }x_i > 0, \\ -1 & \text{if } x_i < 0, \\ 0 & \text{if } x_i = 0. \end{array}\right.$$
Note that we can only hope to recover $\sign(x)$ up to a global sign flip, because $xx^\top = (-x)(-x)^\top$.

\vspace{5pt}
\noindent\begin{minipage}{\linewidth}
\begin{algorithm}[H]
\caption{Recovery of ${\rm supp}(x)$ and ${\rm sign}(x)$ in the spiked Wishart model}
\label{wishart-recovery}
\begin{algorithmic}[1]
\REQUIRE Data $\{y^{(i)}\}_{1\le i\le N}$, parameters $\rho \in (0,1], \beta \ge 0, A \ge 1, \ell\in\mathbb{N}_+$
\STATE $\bar{N}\leftarrow \lfloor N/2 \rfloor$
\STATE Compute sample covariance matrices: $Y' \leftarrow \frac{1}{\bar{N}}\sum_{i = 1}^{\bar{N}} y^{(i)}{y^{(i)}}^\top, Y'' \leftarrow \frac{1}{\bar{N}} \sum_{i = \bar{N}+1}^{2\bar{N}} y^{(i)}{y^{(i)}}^\top$
\STATE Specify the search set: $\sI_{n,\ell} \leftarrow \{v\in \{-1,0,1\}^n \ :\ \|v\|_0 = \ell\}$
\STATE Compute the initial estimate: $v^* \leftarrow {\rm argmax}_{v\in \sI_{n,\ell}}v^{\top} Y' v$
\STATE Compute the refined estimate: $z\leftarrow (Y''-I)v^*$
\FOR{$j = 1$ to $n$}
\STATE $\bar{x}_j \leftarrow \sign(z_j)\cdot \one\{|z_j| > \frac{\beta \ell}{2\sqrt{3}A^2\rho n}\}$
\ENDFOR
\end{algorithmic}
\textbf{Output:} $\bar{x}$
\end{algorithm}
\end{minipage}
\vspace{5pt}

For technical reasons, we divide our $N$ samples into two subsamples of size $\bar{N} = \lfloor N/2 \rfloor$ (with one sample discarded if $N$ is odd) and produce two independent sample covariance matrices $Y'$ and $Y''$. The first step of the algorithm is similar to the detection algorithm: by exhaustive search, we find the vector $v^* \in \sI_{n,\ell}$ maximizing $v^\top Y' v$. In the course of proving that the algorithm succeeds, we will show that $v^*$ has nontrivial correlation with $x$. The second step is to recover the support (and signs) of $x$ by thresholding $z = (Y''-I)v^*$. Note that $z$ discards (i.e., does not depend on) the columns of $Y''$ that do not lie in $\supp(v^*)$; since $\supp(v^*)$ has substantial overlap with $\supp(x)$, this serves to amplify the signal.

\begin{theorem}[Wishart support and sign recovery]\label{rec_wsh}
Consider the planted spiked Wishart model $\PP_n$ with an arbitrary $(\rho,A)$-sparse signal $x$, and let $\gamma = n/N$. Suppose
\begin{equation}\label{eq:wsh-rec-rho}
\rho \le \min\left(1,\frac{\beta}{25 A^8}\right)\frac{A^4\beta}{400\gamma}\frac{1}{\log n}.
\end{equation}
Let $\ell$ be any integer in the interval
\begin{equation}\label{eq:wsh-rec-ell}
\ell \in \left[\frac{10000 A^4\gamma}{\beta^2}\rho^2 n\log n,\; \min\left(1,\frac{\beta}{25 A^8}\right)\frac{25A^8}{\beta}\rho n\right],
\end{equation}
which is nonempty due to \eqref{eq:wsh-rec-rho}.
Then the failure probability of Algorithm~\ref{wishart-recovery} satisfies 
$$1 - \PP_n\left[{\rm sign}(\bar{x}) = \pm {\rm sign}(x)\right]\le 6\exp\left(-\frac{\beta^2}{6400A^4\gamma}\frac{\ell}{\rho^2 n}\right)\le 6n^{-3/2},$$
where the last inequality follows from \eqref{eq:wsh-rec-ell}.
\end{theorem}
\noindent
\begin{remark} As for detection, the runtime of Algorithm~\ref{wishart-recovery} is $n^{O(\ell)}$, and we can minimize this by choosing
$$\ell = \left\lceil\frac{10000 A^4\gamma}{\beta^2}\rho^2 n\log n \right\rceil.$$
\end{remark}

Once we obtain ${\rm supp}(x)$ using Algorithm~\ref{wishart-recovery}, it is straightforward to estimate $x$ (up to global sign flip) using the leading eigenvector of the appropriate submatrix. This step of the algorithm requires only polynomial time.
\begin{theorem}[Wishart recovery]\label{rec2-wsh}
Consider the planted spiked Wishart model $\PP_n$ with an arbitrary $(\rho,A)$-sparse signal $x$, and let $\gamma = n/N$.
Suppose we have access (e.g., via Algorithm~\ref{wishart-recovery}) to $\sI = {\rm supp}(x) \subset [n]$. Write $P_{\sI} = \sum_{i\in \sI}e_i e_i^\top$, $y^{(i)}_{\sI} = P_{\sI}y^{(i)}$ and $Y_{\sI} = P_{\sI}YP_{\sI}^\top = \frac{1}{N}\sum_{i = 1}^{N}y^{(i)}_{\sI}{y^{(i)}_{\sI}}^\top$. Let $\tilde{x}$ denote the unit-norm eigenvector corresponding to the maximum eigenvalue of $Y_{\sI}$. Then, there exists an absolute constant $C>0$ such that, for any $\epsilon \in [\frac{2(1+\beta)\sqrt{\gamma\rho}}{C\beta},1)$,
$$\PP_n \left[\la \tilde{x},x\ra^2 \le 1-\epsilon \right] \le 2\exp\left(-\frac{C^2\beta^2n\epsilon^2}{4(1+\beta)^2\gamma\rho}\right) \le 2 \exp(-n).$$
\end{theorem}
\begin{remark}
In the regime we are interested in, $n \to \infty$ with $A = O(1)$, $\beta = O(1)$, and~\eqref{eq:wsh-rec-rho} is satisfied. In this case, the conclusion of Theorem~\ref{rec2-wsh} gives $\langle \tilde x,x \rangle^2 > 1 - o(1)$ with high probability.
\end{remark}

\paragraph{Low-degree likelihood.} Now, we turn to controlling the low-degree likelihood ratio (LDLR) (see Section~\ref{sec:low-degree}) to provide rigorous evidence that the above algorithms are optimal. In this section we take a fully Bayesian approach, and assume that the planted signal $x$ is drawn from the sparse Rademacher prior $\sX_n^\rho$. Recall that the signal-to-noise ratio is defined as $\hat{\lambda} \colonequals \beta/\sqrt{\gamma}$.

As discussed in Section~\ref{sec:low-degree}, we will determine the behavior of $\|L_n^{\le D}\|$ in the limit $n \to \infty$: if $\|L_n^{\le D}\| = O(1)$, this suggests hardness for $n^{\tilde\Omega(D)}$-time algorithms. We allow the parameters $D,\rho,\beta,\gamma$ to depend on $n$, which we sometimes emphasize by writing, e.g., $\rho_n$. For $D_n = o(n)$, our results suggest hardness for $n^{\tilde\Omega(D)}$-time algorithms whenever $\hat\lambda < 1$ and $\rho \gg \hat\lambda \sqrt{D_n/n}$. This is essentially tight, matching PCA (which succeeds when $\hat\lambda > 1$) and our algorithm with $\ell = D_n$ (which succeeds when $\rho \ll \hat\lambda \sqrt{D_n/n}$) (however, see Remark~\ref{rem:low-deg-fail} below for one caveat).

\begin{theorem} [Boundedness of LDLR for large $\rho$]\label{lowbnd_wsh}
 Under the spiked Wishart model with spike prior $\sX = \sX_n^\rho$, suppose $D_n = o(n)$. If one of the following holds for sufficiently large $n$:
\begin{itemize}
\item[(a)] $\limsup_{n\rightarrow\infty}\hat{\lambda}_n < 1$ and \begin{equation}\label{lowbnd1'}
\rho_n\ge \max\left(1,\sqrt{\frac{1}{6\log(1/\hat{\lambda}_n)}}\right)\sqrt{\frac{D_n}{n}}\text{, or}
\end{equation}
\item[(b)] $\limsup_{n\rightarrow\infty}\hat{\lambda}_n < 1/\sqrt{3}$ and
\begin{equation}\label{lowbnd2'}
\rho_n\ge \hat{\lambda}_n\sqrt{\frac{D_n}{n}},
\end{equation}
\end{itemize}
then, as $n\rightarrow\infty$, $\|L_{n,N,\beta,\sX}^{\le D}\| = O(1)$.
\end{theorem}

The following result on divergence of the LDLR serves as a sanity check: we show that $\|L_n^{\le D}\|$ indeed diverges in the regime where we know that a $n^{\tilde\Omega(D)}$-time algorithm exists.

\begin{theorem} [Divergence of LDLR for small $\rho$]\label{upbnd_wsh}
Under the spiked Wishart model with spike prior $\sX = \sX_n^\rho$, suppose $D_n = \omega(1)$ and $D_n = o(n)$. If one of the following holds:
\begin{itemize}
\item[(a)] $\liminf_{n\rightarrow \infty}\hat{\lambda}_n > 1$\text{, or}
\item[(b)] $\limsup_{n\rightarrow\infty}\hat{\lambda}_n < 1$, $|\log \hat{\lambda}_n| = o(\sqrt{D_n})$ and for sufficiently large $n$,
$$ \rho_n < C\hat{\lambda}_n\log^{-2}(1/\hat{\lambda}_n)\sqrt{\frac{D_n}{n}}$$
where $C$ is an absolute constant,
\end{itemize}
then, as $n\rightarrow\infty$, $\|L_{n,N,\beta,\sX}^{\le D}\| = \omega(1)$.
\end{theorem}

\begin{remark}\label{rem:low-deg-fail}
There is one regime where the above results give some unexpected behavior. Recall first that optimal Bayesian inference for sparse PCA can be performed in time $n^{O(\rho n)}$ by computing the likelihood ratio. Thus if $\|L_n^{\le D}\| = O(1)$ for some $D_n \gg \rho n$, this suggests that the problem is information-theoretically impossible; from our results above, there are regimes where this occurs (and indeed the problem is information-theoretically impossible), yet $\|L_n^{\le D}\| = \omega(1)$ for some larger $D_n$ (which incorrectly suggests that there should be an algorithm). This is analogous to a phenomenon where the second moment of the (non-low-degree) likelihood ratio $\|L_n\|$ can sometimes diverge even when strong detection is impossible (see, e.g.\ \cite{BMNN,BMVVX,PWBM-contig}). Luckily, this issue never occurs for us in the regime of interest $D_n \ll \rho n$, and therefore does not prevent our results from being tight. Note also that none of these observations contradict Conjecture~\ref{conj:low-deg-informal}.
\end{remark}

\subsection{The Wigner Model}
\label{sec:wigner}

We now state our algorithms and results for the Wigner model. These are very similar to the Wishart case, so we omit some of the discussion.

\begin{definition}[Spiked Wigner model]
\label{def:wigner}
The spiked Wigner model with parameters $n\in \NN_+$, $\lambda \ge 0$, and planted signal $x \in \RR^n$ is defined as follows.
\begin{itemize}
\item Under $\PP_n = \PP_{n,\lambda}$, we observe the matrix $Y = W + \lambda xx^\top$, where $W\sim \GOE(n)$.
\item Under $\QQ_n$, we observe the matrix $Y \sim \GOE(n)$.
\end{itemize}
\end{definition}

\vspace{5pt}
\noindent\begin{minipage}{\linewidth}
\begin{algorithm}[H]
\caption{Detection in the spiked Wigner model}
\label{wigner-detection}
\begin{algorithmic}[1]
\REQUIRE Data $Y$, parameters $\rho\in (0,1], \lambda>0, A \ge 1,\ell\in\mathbb{N}_+$
\STATE Specify the search set: $\sI_{n,\ell} \leftarrow \{v\in \{-1,0,1\}^n \ :\ \|v\|_0 = \ell\}$
\STATE Compute the test statistic: $T\leftarrow \max_{v\in \sI_{n,\ell}}v^{\top}Yv$
\STATE Compute the threshold: $T^* \leftarrow \frac{\lambda \ell^2}{2A^2\rho n}$
\IF{$T \ge T^*$}
\RETURN $\tp$
\ELSE
\RETURN $\tq$
\ENDIF
\end{algorithmic}
\end{algorithm}
\end{minipage}
\vspace{5pt}

\begin{remark}[Runtime]
As in the Wishart case (see Remark~\ref{rem:runtime}), the runtime is $n^{O(\ell)}$. The same holds for Algorithm~\ref{wigner-recovery} below.
\end{remark}

\begin{theorem}[Wigner detection]\label{det_wgn}
Consider the spiked Wigner model with an arbitrary $(\rho,A)$-sparse signal $x$. Let $Y$ be drawn from either $\PP_n$ or $\QQ_n$, and let $f_n$ be the output of Algorithm~\ref{wigner-detection}. Suppose
\begin{equation}\label{eq:wgn-det-rho}
\rho \le \frac{\lambda^2}{36A^4}\frac{1}{\log n}.
\end{equation}
Let $\ell$ be any integer in the interval
\begin{equation}\label{eq:wgn-det-ell}
\ell \in \left[\frac{36 A^4}{\lambda^2}\rho^2 n\log n,\; \rho n\right],
\end{equation}
which in nonempty due to \eqref{eq:wgn-det-rho}.
Then the total failure probability of Algorithm~\ref{wigner-detection} satisfies
$$\PP_n[f_n = \tq] + \QQ_n[f_n = \tp] \le 2\exp\left(-\frac{\lambda^2}{32A^4}\frac{\ell^2}{\rho^2 n}\right) \le 2n^{-9 \ell/8},$$
where the last inequality follows from~\eqref{eq:wgn-det-ell}.
\end{theorem}

\begin{remark}\label{rem:alg-lowdeg}
Since our lower bounds are against the class of low-degree algorithms, it is natural to ask whether our algorithms fall into this class. While our test statistic $T$ is not a polynomial function of $Y$, we can instead take as a proxy the degree-$2k$ polynomial $P(Y) = \sum_{v \in \sI_{n,\ell}} (v^\top Y v)^{2k}$ for some choice of $k$. Our analysis can be adapted to show that $P$ can be used to solve strong detection under essentially the same conditions as Theorem~\ref{det_wgn}, provided $k \gtrsim \ell \log n$. Note that (up to log factors) this matches the correspondence between runtime and degree in Conjecture~\ref{conj:low-deg-informal}.
\end{remark}

\vspace{5pt}
\noindent\begin{minipage}{\linewidth}
\begin{algorithm}[H]
\caption{Recovery of ${\rm supp}(x)$ and ${\rm sign}(x)$ in the spiked Wigner model}
\label{wigner-recovery}
\begin{algorithmic}[1]
\REQUIRE Data $Y$, parameters $\rho\in(0,1],\lambda>0,A \ge 1,\ell\in\mathbb{N}_+$
\STATE Sample $\tilde W \sim \GOE(n)$
\STATE Compute independent data matrices: $Y' \leftarrow (Y+\tilde W)/\sqrt{2}$ and $Y'' \leftarrow (Y-\tilde W)/\sqrt{2}$
\STATE Specify the search set: $\sI_{n,\ell} \leftarrow \{v\in \{-1,0,1\}^n \ :\ \|v\|_0 = \ell\}$
\STATE Compute the initial estimate: $v^* \leftarrow {\rm argmax}_{v\in\sI_{n,\ell}}v^{\top}Y'v$
\STATE Compute the refined estimate: $z\leftarrow Y'' v^*$
\FOR{$j = 1$ to $n$}
\STATE $\bar{x}_j = \sign(z_j)\cdot \one\{|z_j| > \frac{\lambda \ell}{4A^2\rho n}\}$
\ENDFOR
\end{algorithmic}
\textbf{Output:} $\bar{x}$
\end{algorithm}
\end{minipage}
\vspace{5pt}

\noindent
For technical reasons, our first step is to fictitiously ``split'' the data into two independent copies $Y'$ and $Y''$. Note that
$$Y' = \frac{\lambda}{\sqrt{2}}xx^{\top}+ \frac{W+\tilde W}{\sqrt{2}}\quad\text{and}\quad Y'' = \frac{\lambda}{\sqrt{2}}xx^{\top}+ \frac{W-\tilde W}{\sqrt{2}}.$$
Since $W' \colonequals \frac{W+\tilde W}{\sqrt{2}}$ and $W'' \colonequals \frac{W-\tilde W}{\sqrt{2}}$ are independent $\GOE(n)$ matrices, $Y'$ and $Y''$ are distributed as independent observations drawn from $\PP_n$ with the same planted signal $x$ and with effective signal-to-noise ratio $\bar{\lambda} = \lambda/\sqrt{2}$. 

\begin{theorem}[Wigner support and sign recovery]\label{rec_wgn}
Consider the planted spiked Wishart model $\PP_n$ with an arbitrary $(\rho,A)$-sparse signal $x$. Suppose
\begin{equation}\label{eq:wgn-rec-rho}
\rho \le \frac{\lambda^2}{338A^4}\frac{1}{\log n}.
\end{equation}
Let $\ell$ be any integer in the interval
\begin{equation}\label{eq:wgn-rec-ell}
\ell \in \left[\frac{338 A^4}{\lambda^2}\rho^2 n\log n,\; \rho n\right],
\end{equation}
which is nonempty due to \eqref{eq:wgn-rec-rho}. Then the failure probability of Algorithm~\ref{wigner-recovery} satisfies 
$$1-\PP_n\left[{\rm supp}(\bar{x}) = {\rm supp}(x),\ {\rm sign}(\bar{x}) = \pm {\rm sign}(x)\right)\le 4\exp\left(-\frac{\lambda^2}{288A^4}\frac{\ell}{\rho^2 n}\right)\le 4n^{-169/144},$$
where the last inequality follows from \eqref{eq:wgn-rec-ell}.
\end{theorem}

As in the Wishart case, once we have recovered the support, there is a standard polynomial-time spectral method to estimate $x$.
\begin{theorem}[Wigner recovery]\label{rec2-wgn}
Consider the planted spiked Wigner model $\PP_n$ with an arbitrary $(\rho,A)$-sparse signal $x$. Suppose we have access (e.g., via Algorithm~\ref{wigner-recovery}) to $\sI = {\rm supp}(x) \subset [n]$. Write $P_{\sI} = \sum_{i\in \sI}e_i e_i^\top$ and $Y_{\sI} = P_{\sI}YP_{\sI}^\top$. Let $\tilde{x}$ denote the unit-norm eigenvector corresponding to the maximum eigenvalue of $Y_{\sI}$. Then for any $\epsilon \in (\frac{4\sqrt{2\rho}}{\lambda},1)$, 
$$
    \PP_n\left[\la \tilde{x},x\ra^2 \le 1-\epsilon \right]
    \le 4\exp\left[-\frac{n}{16}\left(\lambda\epsilon-4\sqrt{2\rho}\right)^2\right].
$$
\end{theorem}
\begin{remark}
In the regime we are interested in, $n \to \infty$ with~\eqref{eq:wgn-rec-rho} satisfied, so that $\sqrt{\rho}/\lambda \rightarrow 0$. In this case, the conclusion of Theorem~\ref{rec2-wgn} gives $\langle \tilde x,x \rangle^2 > 1 - o(1)$ with high probability, upon choosing for example $\epsilon = \frac{8\sqrt{2\rho}}{\lambda}$.
\end{remark}

\noindent
We also have the following results on the behavior of the low-degree likelihood ratio.
\begin{theorem} [Boundedness of LDLR for large $\rho$]\label{lowbnd_wgn}
 Under the spiked Wigner model with prior $\sX = \sX_n^\rho$, suppose $D_n = o(n)$. If one of the following holds for sufficiently large $n$:
\begin{itemize}
\item[(a)] $\limsup_{n\rightarrow\infty}\lambda_n < 1$ and \begin{equation}\label{lowbnd1}
\rho_n\ge \max\left(1,\sqrt{\frac{1}{6\log(1/\lambda_n)}}\right)\sqrt{\frac{D_n}{n}}\text{, or}
\end{equation}
\item[(b)] $\limsup_{n\rightarrow\infty}\lambda_n < 1/\sqrt{3}$ and
\begin{equation}\label{lowbnd2}
\rho_n\ge \lambda_n\sqrt{\frac{D_n}{n}},
\end{equation}
\end{itemize}
then, as $n\rightarrow\infty$, $\|L_{n,\lambda,\sX}^{\leq D}\| = O(1)$.
\end{theorem}
\begin{theorem} [Divergence of LDLR for small $\rho$]\label{upbnd_wgn}
 Under the spiked Wigner model with prior $\sX = \sX_n^\rho$, suppose $D_n = \omega(1)$ and $D_n = o(n)$. If one of the following holds:
\begin{itemize}
\item[(a)] $\liminf_{n\rightarrow \infty}\lambda_n > 1$\text{, or}
\item[(b)] $\limsup_{n\rightarrow\infty}\lambda_n < 1$, $|\log \lambda_n| = o(\sqrt{D_n})$ and for sufficiently large $n$,
$$ \rho_n < C\lambda_n\log^{-2}(1/\lambda_n)\sqrt{\frac{D_n}{n}}$$
where $C$ is an absolute constant,
\end{itemize}
then, as $n\rightarrow\infty$, $\|L_{n,\lambda,\sX}^{\leq D}\| = \omega(1)$.
\end{theorem}

\section{Proofs for Subexponential-Time Algorithms}
\label{sec:proofs-alg}

\subsection{The Wishart Model}
\begin{proof}[Proof of Theorem \ref{det_wsh} (Detection)]
Under $\mathbb{Q}_n$, for any fixed $v\in \sI_{n,\ell}$ we have $v^\top y^{(i)} \sim \sN(0,\ell)$ for $i\in [N]$ and
$$v^\top Yv = \frac{1}{{N}}\sum_{i = 1}^{{N}} (v^\top y^{(i)})^2 \eqd \frac{\ell}{{N}}\chi_{{N}}^2,$$
where $\chi_N^2$ is a chi-squared random variable with $N$ degrees of freedom, i.e., the sum of the squares of $N$ standard gaussians.
Using Corollary \ref{coro-cher}, we union bound over $v\in  \sI_{n,\ell}$ for any $t\in (0,\frac{1}{2})$:
\begin{equation*}
\begin{aligned}
\mathbb{Q}_n\left[T\ge\ell(1+t)\right] &\le | \sI_{n,\ell}|\Pr\left[\chi_{{N}}^2 \ge {N}(1+t)\right]\\
&\le \binom{n}{\ell}2^\ell\cdot \exp\left(-\frac{{N}t^2}{3}\right)\\
&\le \exp\left(\ell\log(2n) -\frac{{N}t^2}{3}\right).\\
\end{aligned}
\end{equation*}
Under the condition 
\begin{equation}\label{condq}
{N}t^2 \ge 4 \ell\log (2n)
\end{equation}
we have
$$\mathbb{Q}_n\left[T\ge\ell(1+t)\right] \le \exp\left(-\frac{1}{12}{N}t^2\right).$$
Meanwhile, under $\mathbb{P}_n$, when $v = \bar{v}$ correctly guesses $\ell$ entries and their signs in the support of $x$ (which requires $\ell \le \rho n$), for any $i\in [N]$ we have
$${\bar{v}}^\top y^{(i)} \sim \sN(0,{\bar{v}}^\top (I_n+\beta xx^\top)\bar{v}) = \sN(0,\ell+\beta \la \bar{v},x\ra^2).$$
Therefore,
$${\bar{v}}^\top Y\bar{v} = \frac{1}{{N}}\sum_{i = 1}^{{N}} ({\bar{v}}^\top y^{(i)})^2 \eqd \frac{1}{{N}}(\ell+\beta \la \bar{v},x\ra^2)\chi_{{N}}^2$$
where $\la \bar{v},x\ra^2 \ge \frac{\ell^2}{A^2\rho n}$.
As a result, by Corollary \ref{coro-cher},
\begin{equation*}
\begin{aligned}
\mathbb{P}_n\left[T<\ell(1+t)\right] &\le \mathbb{P}_n\left[{\bar{v}}^\top Y\bar{v} < \ell(1+t)\right]\\
&= \Pr\left[\frac{1}{{N}}(\ell+\beta \la \bar{v},x\ra^2)\chi_{{N}}^2 < \ell(1+t)\right]\\
&= \Pr\left[\chi_{{N}}^2 < {N}\left(1-\frac{\beta \ell-A^2\rho n t}{\beta \ell +A^2\rho n}\right)\right]\\
&\le \exp\left(-\frac{{N}}{3}\left(\frac{\beta \ell-A^2\rho n t}{\beta \ell +A^2\rho n}\right)^2\right),
\end{aligned}
\end{equation*}
the last inequality requiring
\begin{equation}\label{condp}
0 \le \frac{\beta\ell-A^2\rho nt}{\beta\ell +A^2\rho n} \le \frac{1}{2}.
\end{equation}
To satisfy $t\in (0,\frac{1}{2})$, \eqref{condq} and \eqref{condp} at the same time, we choose 
$$t = \frac{\beta \ell}{2A^2\rho n}.$$
Under the condition
\begin{equation}\label{finalcond}
\frac{\beta \ell}{A^2 n} \le \rho \le \frac{\beta}{5A^2\sqrt{\gamma}}\sqrt{\frac{\ell}{n\log n}},
\end{equation}
which is equivalent to the interval for $\ell$ given in \eqref{eq:wsh-det-ell}, thresholding the statistic $T$ at $\ell(1+t)$ succeeds at distinguishing $\PP_n$ and $\QQ_n$ with total error probability 
$$\mathbb{Q}_n\left[T\ge\ell(1+t)\right]
+\mathbb{P}_n\left[T<\ell(1+t)\right] \le 2\exp\left(-\frac{\beta^2}{48A^4\gamma}\frac{\ell^2}{\rho^2 n}\right),$$
which completes the proof.
\end{proof}

\begin{proof}[Proof of Theorem \ref{rec_wsh} (Support and Sign Recovery)]
First, we give a high-probability lower bound on $\la v^*,x\ra$. From the analysis of the detection algorithm, we know that under the condition \eqref{finalcond},
\begin{align*}
&1-2\exp\left(-\frac{\beta^2}{48A^4\gamma}\frac{\ell^2}{\rho^2 n}\right) \\
&\hspace{2cm}\le \PP_n \left[\ell\left(1+\frac{\beta \ell}{2A^2\rho n}\right) \le {v^*}^\top Y' v^*\right] \\
&\hspace{2cm}\le \PP_n\left[\ell\left(1+\frac{\beta \ell}{2A^2\rho n}\right) \le \frac{1}{{\bar{N}}}(\ell+\beta \la v^*,x\ra^2)\chi_{\bar{N}}^2,\ \la v^*,x\ra^2 < \frac{\ell^2}{3A^2\rho n} \right] \\
&\hspace{3cm}+\PP_n\left[\ell\left(1+\frac{\beta \ell}{2A^2\rho n}\right) \le \frac{1}{{\bar{N}}}(\ell+\beta \la v^*,x\ra^2)\chi_{\bar{N}}^2,\ \la v^*,x\ra^2 \ge \frac{\ell^2}{3A^2\rho n}\right] \\
&\hspace{2cm}\le \PP_n\left[\ell\left(1+\frac{\beta \ell}{2A^2\rho n}\right) \le \frac{1}{{\bar{N}}}\left(\ell+\frac{\beta\ell^2}{3A^2\rho n}\right)\chi_{\bar{N}}^2\right] + \PP_n\left[\la v^*,x\ra^2 \ge \frac{\ell^2}{3A^2\rho n}\right],
\end{align*}
where
\begin{align*}
\PP_n\left[\ell\left(1+\frac{\beta \ell}{2A^2\rho n}\right) \le \frac{1}{{\bar{N}}}\left(\ell+\frac{\beta\ell^2}{3A^2\rho n}\right)\chi_{\bar{N}}^2\right] &=
\PP_n\left[\chi_{\bar{N}}^2 \ge {\bar{N}}\left(1+\frac{\beta\ell}{2\beta\ell +6A^2\rho n}\right)\right]\\
&\le \exp\left(-\frac{{\bar{N}}}{3}\left(\frac{\beta\ell}{2\beta\ell+6A^2\rho n}\right)^2\right)\\
&\le \exp\left(-\frac{\beta^2}{384A^4\gamma}\frac{\ell^2}{\rho^2 n}\right),
\end{align*}
hence we have the lower bound
\begin{align*}
        \PP_n\left[\la v^*,x\ra^2 \ge \frac{\ell^2}{3A^2\rho n}\right] &\ge 1-2\exp\left(-\frac{\beta^2}{48A^4\gamma}\frac{\ell^2}{\rho^2 n}\right) -  \exp\left(-\frac{\beta^2}{384A^4\gamma}\frac{\ell^2}{\rho^2 n}\right)\\
        &\ge 1-3\exp\left(-\frac{\beta^2}{384A^4\gamma}\frac{\ell^2}{\rho^2 n}\right).
    \end{align*}
We now fix $v^*$ satisfying the above lower bound on $\la v^*,x \ra^2$. From this point onward, we will only use the second copy $Y''$ of our data; note that, crucially, $Y''$ is independent from $v^*$. To simplify the notation, we will write $y^{(1)},\ldots,y^{(\bar{N})}$ instead of $y^{(\bar{N}+1)},\ldots,y^{(2\bar{N})}$ for the samples used to form $Y''$. We now adopt an equivalent representation of the observations: $y^{(i)} = u^{(i)} + \sqrt{\beta} w^{(i)} x$, where $u^{(i)}\sim \sN (0,I_n)$ and $w^{(i)} \sim \sN (0,1)$ are independent random gaussian vectors and scalars, respectively. Substituting this into $z = (Y''-I)v^*$ yields
 $$z_{j} = \frac{1}{{\bar{N}}}\sum_{i = 1}^{\bar{N}}(a_{ij}+b_{ij}+c_{ij} +d_{ij} +e_{ij})$$
where, for $i \in [\bar{N}]$ and $j\in [n]$,
\begin{align*}
a_{ij} &= (w^{(i)})^2 \beta x_{j} \la v^*,x\ra\\
b_{ij} &= ((u^{(i)}_j)^2-1)v^*_j\\
c_{ij} &= \sum_{k\neq j}u^{(i)}_j u^{(i)}_k v^*_k\\
d_{ij} &= u^{(i)}_j w^{(i)} \sqrt{\beta} \la v^*,x\ra\\
e_{ij} &= \la u^{(i)}, v^* \ra w^{(i)}\sqrt{\beta} x_{j},
\end{align*}
with $\mathbb{E}(a_{ij}) = \beta x_{j}\la v^*,x\ra$ and $\mathbb{E}(b_{ij}) = \mathbb{E}(c_{ij}) = \mathbb{E}(d_{ij}) = \mathbb{E}(e_{ij}) = 0$. We will show separate union bounds for these five contributions to $z_j$. In the following, we fix the constant $\mu = 1/20$.

\paragraph{Union bound for $a_{ij}$.}
For all $j\in {\rm supp}(x)$,
$$\frac{\beta\ell}{\sqrt{3} A^2\rho n} \le \beta \frac{1}{A\sqrt{\rho n}}\frac{
\ell}{\sqrt{3}A\sqrt{\rho n}} \le |\beta x_{j} \la v^*,x\ra| \le \beta \frac{A}{\sqrt{\rho n}}\frac{A\ell}{\sqrt{\rho n}} = \frac{A^2\beta \ell}{\rho n},$$
so by Corollary \ref{coro-cher},
\begin{equation*}
\begin{aligned}
\log\PP_n\left[\left|\frac{1}{\bar{N}}\sum_{i = 1}^{\bar{N}} a_{ij}-\beta x_{j} \la v^*,x\ra\right| > \mu\frac{\beta \ell}{A^2\rho n}\right] &= 
\log\Pr\left[\left|\frac{1}{\bar{N}}\chi_{\bar{N}}^2 -1\right| > \mu\frac{\beta \ell}{A^2\rho n}\cdot \frac{1}{|\beta x_{j} \la v^*,x\ra|}\right]\\
&\le \log\Pr\left[\left|\frac{1}{\bar{N}}\chi_{\bar{N}}^2 -1\right| > \frac{\mu}{A^4}\right]\\
&\le -\frac{\mu^2 \bar{N}}{3A^8}.
\end{aligned}
\end{equation*}
Therefore, we may union bound over $j\in {\rm supp}(x)$:
\begin{align}\label{uniona}
&\PP_n\left[\left|\frac{1}{{\bar{N}}}\sum_{i = 1}^{\bar{N}} a_{ij}-\beta x_{j} \la v^*,x\ra\right| \le \mu\frac{\beta \ell}{A^2\rho n}, \text{ for all } j\in [n]\right]\nonumber\\
&\hspace{2cm}\ge 1-\sum_{j\in {\rm supp}(x)}\PP_n\left[\left|\frac{1}{{\bar{N}}}\sum_{i = 1}^{\bar{N}} a_{ij}-\beta x_{j} \la v^*,x\ra\right| > \mu\frac{\beta \ell}{A^2\rho n}\right]\nonumber\\
&\hspace{2cm}\ge 1-\exp\left(\log n-\frac{\mu^2 {\bar{N}}}{3A^8}\right)\nonumber\\
&\hspace{2cm}\ge 1-\exp\left(-\frac{\mu^2 n}{7A^8\gamma}\right).
\end{align}

\paragraph{Union bound for $b_{ij}$.}
We have $b_{ij}$ nonzero only when $j\in {\rm supp}(v^*)$. For such $j$, by Corollary~\ref{coro-cher},
\begin{equation*}
\begin{aligned}
\log\PP_n\left[\left|\frac{1}{{\bar{N}}}\sum_{i = 1}^{\bar{N}} b_{ij}\right| > \mu\frac{\beta \ell}{A^2\rho n}\right] &= 
\log\Pr\left[\left|\frac{1}{{\bar{N}}}\chi_{\bar{N}}^2 -1\right| > \mu\frac{\beta \ell}{A^2\rho n}\right]\\
&\le -\frac{{\bar{N}}}{3}\left(\mu\frac{\beta \ell}{A^2\rho n}\right)^{2}.
\end{aligned}
\end{equation*}
 Therefore,
\begin{align}\label{unionb}
\PP_n\left[\left|\frac{1}{{\bar{N}}}\sum_{i = 1}^{\bar{N}} b_{ij}\right| \le \mu\frac{\beta \ell}{A^2\rho n}, \text{ for all } j\in [n]\right] &\ge 1-\sum_{j\in {\rm supp}(v^*)}\PP_n\left[\left|\frac{1}{{\bar{N}}}\sum_{i = 1}^{\bar{N}} b_{ij}\right| > \mu\frac{\beta \ell}{A^2\rho n}\right]\nonumber\\
&\ge 1-\exp\left(\log \ell -\frac{{\bar{N}}}{3}\left(\mu\frac{\beta \ell}{A^2\rho n}\right)^2\right)\nonumber\\
&\ge 1-\exp\left(-\frac{\mu^2\beta^2}{12A^4\gamma}\frac{\ell^2}{\rho^2 n}\right),
\end{align}
under the condition
$$\rho \le \frac{1}{\sqrt{12}} \frac{\mu\beta}{A^2\sqrt{\gamma}}\sqrt{\frac{\ell}{n\log \ell}}.$$

\paragraph{Union bound for $c_{ij}$ and $d_{ij}$.}
In the following, let $u, u'$ denote independent samples from $\sN(0,I_{\bar{N}})$. Note that $$\tilde{u}_j^{(i)} \colonequals \sum_{k\neq j} u^{(i)}_k v^*_k \sim \sN(0,\tilde{\ell}_j), \text{ where } \tilde{\ell}_j = \ell-\one\{j\in {\rm supp}(v^*)\},$$
and $\tilde{u}_j^{(i)}$ is independent from $u^{(i)}_j$. 
Therefore,
$$\sum_{i = 1}^{\bar{N}} c_{ij} = \sum_{i = 1}^{\bar{N}} u^{(i)}_j \tilde{u}_j^{(i)} \eqd \sqrt{\tilde{\ell}_j}\la u,u'\ra.$$
Therefore, by Lemma \ref{prodest}, for the $c_{ij}$ we have
\begin{align}\label{unionc}
\PP_n\left[\left|\frac{1}{{\bar{N}}}\sum_{i = 1}^{\bar{N}} c_{ij}\right| \le \mu\frac{\beta \ell}{A^2\rho n}, \text{ for all } j\in [n]\right] &\ge 1-\sum_{j = 1}^n\Pr\left[|\la u,u'\ra| > \mu\frac{\beta \sqrt{\ell}{\bar{N}}}{A^2\rho n}\right]\nonumber\\
&\ge 1-2n\exp\left(-\frac{1}{4{\bar{N}}}\left(\mu\frac{\beta \sqrt{\ell}{\bar{N}}}{A^2\rho n}\right)^2\right)\nonumber\\
&\ge 1-\exp\left(-\frac{\mu^2\beta^2}{16A^4\gamma}\frac{\ell}{\rho^2 n}\right).
\end{align}
The last inequality holds under the condition
$$\mu\frac{\beta \sqrt{\ell}}{A^2\rho n}\le \frac{1}{2},\ \frac{\mu^2\beta^2}{16A^4\gamma}\frac{\ell}{\rho^2 n} \ge \log (2n),$$ which follows from $$\frac{2\mu\beta\sqrt{\ell}}{A^2 n} \le \rho \le \frac{\mu\beta}{5A^2\sqrt{\gamma}}\sqrt{\frac{\ell}{n\log n}}.$$

Meanwhile,
$$\sum_{i = 1}^{\bar{N}} d_{ij} \eqd \sqrt{\beta}\la v^*,x\ra \la u,u'\ra.$$
Therefore, as for the $c_{ij}$, for the $d_{ij}$ we have
\begin{align}\label{uniond}
\PP_n\left[\left|\frac{1}{{\bar{N}}}\sum_{i = 1}^{\bar{N}} d_{ij}\right| \le \mu\frac{\beta \ell}{A^2\rho n}, \text{ for all } j\in [n]\right] 
&\ge 1-\sum_{j = 1}^n\Pr\left[|\la u,u'\ra| > \mu\frac{\sqrt{\beta}\bar{N}}{A^3\sqrt{\rho n}}\right]\nonumber\\
&\ge 1-2n\exp\left(-\frac{1}{4{\bar{N}}}\left(\mu\frac{\sqrt{\beta}\bar{N}}{A^3\sqrt{\rho n}}\right)^2\right)\nonumber\\
&\ge 1-\exp\left(-\frac{\mu^2\beta}{16A^6\gamma}\frac{1}{\rho}\right).
\end{align}
The last inequality holds under the condition
$$\mu\frac{\sqrt{\beta}}{A^3\sqrt{\rho n}} \le \frac{1}{2},\ \frac{\mu^2\beta}{16A^6\gamma}\frac{1}{\rho} \ge \log(2n), $$ which follows from $$\frac{4\mu^2\beta}{A^6 n}\le \rho \le \frac{\mu^2\beta}{17A^6\gamma \log n}.$$

\paragraph{Union bound for $e_{ij}$.}
We have
$$\sum_{i = 1}^{\bar{N}} e_{ij} \eqd x_j\sqrt{\beta\ell}  \la u,u' \ra, $$
which is only nonzero for $j\in {\rm supp}(x)$. Therefore,
\begin{align}\label{unione}
\PP_n\left[\left|\frac{1}{{\bar{N}}}\sum_{i = 1}^{\bar{N}} e_{ij}\right| \le \mu\frac{\beta \ell}{A^2\rho n},\text{ for all } j\in [n]\right] &\ge 1-\sum_{j\in {\rm supp}(x)}\Pr\left[|\la u,u'\ra| > \mu \frac{\sqrt{\beta\ell}\bar{N}}{A^2\rho n |x_j|}\right]\nonumber\\
&\ge 1-\sum_{j\in {\rm supp}(x)}\Pr\left[|\la u,u'\ra| > \mu \frac{\sqrt{\beta\ell}\bar{N}}{A^4\sqrt{\rho n}}\right]\nonumber\\
&\ge 1-2n\exp\left(-\frac{1}{4{\bar{N}}}\left(\mu\frac{\sqrt{\beta\ell}\bar{N}}{A^4\sqrt{\rho n}}\right)^2\right)\nonumber\\
&\ge 1-\exp\left(-\frac{\mu^2\beta}{16A^8\gamma}\frac{\ell}{\rho}\right).
\end{align}
The last inequality holds under the condition
$$\mu\frac{\sqrt{\beta\ell}}{A^4\sqrt{\rho n}} \le \frac{1}{2},\  \frac{\mu^2\beta}{16A^8\gamma}\frac{\ell}{\rho} \ge \log (2n), $$
which follows from
$$\frac{4\mu^2\beta \ell}{A^8 n}\le \rho \le \frac{\mu^2\beta \ell}{17A^8\gamma \log n}.$$

\paragraph{Final steps.}
Now, combining all of the union bounds and conditions from \eqref{uniona}, \eqref{unionb}, \eqref{unionc}, \eqref{uniond} and \eqref{unione}, assuming that $\beta, \gamma = \Theta(1)$ and that $\omega(1) \leq \ell(n) \leq o(n/\log n)$, under the condition
\begin{equation}\label{condrho}
\max\left(1,\frac{\beta}{25A^8}\right)\frac{\ell}{n} \le \rho \le \frac{\beta}{100A^2\sqrt{\gamma}}\sqrt{\frac{\ell}{n\log n}}
\end{equation}
which is equivalent to the regime for $\ell$ given in \eqref{eq:wsh-rec-ell} that we are considering, we have
\begin{align*}
&\PP_n\left[\text{for some } j, \text{ }\left|z_{j}-\beta x_{j} \la v^*,x\ra\right| > \frac{\beta \ell}{4A^2\rho n}\right] + \PP_n\left[\la v^*,x\ra^2 \le \frac{\ell^2}{3A^2\rho n}\right] \\
&\hspace{2cm}\le \PP_n\left[\text{for some } j, \text{ } \left|\frac{1}{\bar{N}}\sum_{i = 1}^{\bar{N}} a_{ij}-\beta x_{j} \la v^*,x\ra\right| > \mu\frac{\beta \ell}{A^2\rho n}\right] \\
&\hspace{3cm} +\PP_n\left[\text{for some } j, \text{ } \left|\frac{1}{\bar{N}}\sum_{i = 1}^{\bar{N}} b_{ij}\right| > \mu\frac{\beta \ell}{A^2\rho n}\right]\\
&\hspace{3cm} +\PP_n\left[\text{for some } j, \text{ } \left|\frac{1}{\bar{N}}\sum_{i = 1}^{\bar{N}} c_{ij}\right| > \mu\frac{\beta \ell}{A^2\rho n}\right] \\
&\hspace{3cm} +\PP_n\left[\text{for some } j, \text{ } \left|\frac{1}{\bar{N}}\sum_{i = 1}^{\bar{N}} d_{ij}\right| > \mu\frac{\beta \ell}{A^2\rho n}\right]\\
&\hspace{3cm}+\PP_n\left[\text{for some } j, \text{ } \left|\frac{1}{\bar{N}}\sum_{i = 1}^{\bar{N}} e_{ij}\right| > \mu\frac{\beta \ell}{A^2\rho n}\right] \\
&\hspace{3cm} +\PP_n\left[\la v^*,x\ra^2 \le \frac{\ell^2}{3A^2\rho n}\right]\\
&\hspace{2cm}\le 5\exp\left(-\frac{\beta^2}{6400A^4\gamma}\frac{\ell}{\rho^2 n}\right) +\exp\left(-\frac{\beta^2}{384A^4\gamma}\frac{\ell^2}{\rho^2 n}\right) \\ 
&\hspace{2cm}\le 6\exp\left(-\frac{\beta^2}{6400A^4\gamma}\frac{\ell}{\rho^2 n}\right).
\end{align*}
Since $|\beta x_{j} \la v^*,x\ra| \ge \frac{\beta \ell}{\sqrt{3}A^2\rho n}$ for $j\in {\rm supp}(x)$ and $|\beta x_{j} \la v^*,x\ra| = 0$ for $j\notin {\rm supp}(x)$, we conclude that, with probability at least $ 1-6\exp\left(-\frac{\beta^2}{6400A^4\gamma}\frac{\ell}{\rho^2 n}\right)$, for every $j\in [n]$, 
$$j\in {\rm supp}(x) \hspace{1em}\text{ if and only if }\hspace{1em} |z_{j}|\ge \frac{\beta \ell}{2\sqrt{3}A^2\rho n},$$
and
$${\rm sign}(z_j) = {\rm sign}(x_j\la v^*,x\ra),$$
completing the proof.
\end{proof}

\begin{proof} [Proof of Theorem \ref{rec2-wsh} (Full Recovery)]
By a result in the analysis of covariance matrix estimation for subgaussian distributions (\cite{vershynin-intro}, Remark 5.51), there exists an absolute constant $C>0$ such that, for any $\delta\in [\sqrt{\rho\gamma}/C,1)$, the following holds with probability at least
$1-2\exp\left(-\frac{C^2\delta^2 N}{\rho}\right)$:
$$\|Y_{\sI}-(P_{\sI}+\beta xx^\top)\| \le \delta \|P_{\sI}+\beta xx^\top\| = \delta(1+\beta).$$
Whenever this is true, by the definition of spectral norm we have
\begin{align*}
    \delta(1+\beta) &\ge \tilde{x}^\top \left[Y_{\sI}-(P_{\sI}+\beta xx^\top)\right]\tilde{x}\\
    &= \|Y_{\sI}\|-(1+\beta \la \tilde{x},x \ra^2)\\
    &\ge (1-\delta)(1+\beta)-(1+\beta \la \tilde{x},x \ra^2),
\end{align*}
which is equivalent to $\la \tilde{x},x\ra^2 \ge 1-\epsilon$ upon taking $\delta = \frac{\beta\epsilon}{2(1+\beta)}$. Thus, for any $\epsilon \in (\frac{2(1+\beta)\sqrt{\rho\gamma}}{C\beta},1)$,
\[ \Pr \left[\la \tilde{x},x\ra^2 \le 1-\epsilon \right] \le 2\exp\left(-\frac{C^2\beta^2n\epsilon^2}{4(1+\beta)^2\gamma\rho}\right), \]
which completes the proof.
\end{proof}

\subsection{The Wigner Model}
\begin{proof}[Proof of Theorem \ref{det_wgn} (Detection).]
For simplicity we denote $t = \frac{\lambda\ell^2}{2A^2\rho n}$. Under $\mathbb{P}_n$, when ${\bar{v}}$ correctly guesses $\ell$ entries in the support of $x$ with correct signs (which requires $\ell \le \rho n$),
\[ {{\bar{v}}}^\top Y {\bar{v}} = {{\bar{v}}}^\top W{\bar{v}} + \lambda \la {\bar{v}},x\ra^2, \]
where ${{\bar{v}}}^\top W{\bar{v}} \sim \mathcal{N}(0,\ell^2/n)$.
Note that 
\[ \lambda \la {\bar{v}},x\ra^2 \ge \frac{\lambda \ell^2}{A^2\rho n}=2t. \]
Therefore, a standard Gaussian tail bound gives
\begin{align*}
\PP_n\left[T < t\right] &\le \PP_n\left[{{\bar{v}}}^\top Y{\bar{v}} < t\right] \\ &\le \Pr\left[\mathcal{N}(0,\ell^2/n) > t\right] \\
&\le \exp\left(-\frac{n}{2\ell^2} \left(\frac{\lambda \ell^2}{2A^2\rho n}\right)^2\right) \\ &= \exp\left(-\frac{\lambda^2}{8A^4}\frac{\ell^2}{\rho^2 n}\right).
\end{align*}
Under $\mathbb{Q}_n$, for each fixed $v \in \mathcal{I}_{n,\ell}$, we have
$$v^\top Y v \sim \mathcal{N}(0,2\ell^2/n).$$
By the same tail bound,
$$\QQ_n\left[v^\top Y v \ge t\right] \le \exp\left(-\frac{nt^2}{4\ell^2}\right).$$
Now, by a union bound over $v \in \mathcal{I}_{n,\ell}$,
\begin{align*}
\QQ_n\left[T \ge t\right] &\le |\mathcal{I}_{n,\ell}| \exp\left(-\frac{nt^2}{4\ell^2}\right)\\
&= \binom{n}{\ell} 2^\ell \exp\left(-\frac{nt^2}{4\ell^2}\right)\\
&\le \exp\left(\ell \log(2n) -\frac{nt^2}{4\ell^2}\right).
\end{align*}
Under the condition
$$\frac{nt^2}{8\ell^2}\ge \ell\log(2n)\ \Leftarrow\ \rho < \frac{\lambda}{6A^2}{\sqrt{\frac{\ell}{n\log n}}},$$
which is equivalent to the interval for $\ell$ given in \eqref{eq:wgn-det-ell}, we have
$$\QQ_n\left[T \ge t\right] \le \exp\left(-\frac{nt^2}{8\ell^2}\right) = \exp\left(-\frac{\lambda^2}{32A^4}\frac{\ell^2}{\rho^2 n}\right).$$
Therefore, by thresholding $T$ at $t$, under the condition 
\begin{equation}\label{wig-cond}
\frac{\ell}{n}\le \rho \le \frac{\lambda}{6A^2} \sqrt{\frac{\ell}{n \log n}},
\end{equation}
we can distinguish $\PP_n$ and $\QQ_n$ with total failure probability at most
$$\PP_n\left[T < t\right] + \QQ_n\left[T \ge t\right] \le \exp\left(-\frac{\lambda^2}{8A^4}\frac{\ell^2}{\rho^2 n}\right)+\exp\left(-\frac{\lambda^2}{32A^4}\frac{\ell^2}{\rho^2 n}\right)\le 2\exp\left(-\frac{\lambda^2}{32A^4}\frac{\ell^2}{\rho^2 n}\right),$$
completing the proof.
\end{proof}

\begin{proof}[Proof of Theorem \ref{rec_wgn} (Support and Sign Recovery).]

First, we show that $v^*$ has significant overlap with the support of $x$. From the analysis of the detection algorithm, provided \eqref{wig-cond} holds, with probability at least $1-2\exp\left(-\frac{\bar{\lambda}^2}{32A^4}\frac{\ell^2}{\rho^2 n}\right)$ we have
$$\frac{\bar{\lambda}\ell^2}{2A^2\rho n} \le {v^*}^\top Y' v^* = \bar{\lambda} \langle v^*,x \rangle^2 + {v^*}^\top W' v^*.$$
where ${v^*}^\top W' v^* \sim \sN(0,2\ell^2/n)$.
Therefore, for $n$ sufficiently large, 
\begin{align*}
\PP_n\left[\langle v^*,x \rangle^2 \ge \frac{\ell^2}{4A^2\rho n}\right] &\ge \left(1-2\exp\left(-\frac{\bar{\lambda}^2}{32A^4}\frac{\ell^2}{\rho^2 n}\right)\right) \left(1- \Pr\left[\sN(0,2\ell^2/n) \ge \frac{\bar{\lambda}\ell^2}{4A^2\rho n}\right]\right)\\
&\ge 1-2\exp\left(-\frac{\bar{\lambda}^2}{32A^4}\frac{\ell^2}{\rho^2 n}\right)-\exp\left(-\frac{\bar{\lambda}^2}{64A^4}\frac{\ell^2}{\rho^2 n}\right)\\
&\ge 1-3\exp\left(-\frac{\bar{\lambda}^2}{64A^4}\frac{\ell^2}{\rho^2 n}\right).
\end{align*}
We now fix $v^*$ satisfying the above lower bound on $\langle v^*,x \rangle^2$. From this point onward, we will only use the second copy $Y''$ of our data; it is important here that $Y''$ is independent from $v^*$. We will that $x$ is successfully recovered by thresholding the entries of $z = Y''v^*$. Entrywise, we have
$$z_i = \bar{\lambda} x_i \langle v^*,x \rangle + e_i^\top W''v^*.$$
For all $i \in \mathrm{supp}(x)$,
$$|\bar{\lambda} x_i \langle v^*,x \rangle| \ge \bar{\lambda} \frac{1}{A\sqrt{\rho n}}\cdot \frac{\ell}{2A\sqrt{\rho n}} = \frac{\bar{\lambda} \ell}{2A^2\rho n}.$$
For simplicity we denote $s = \frac{\bar{\lambda}\ell}{2A^2\rho n}$ and $\mu = \frac{1}{3}$. Note that for all $i\in [n]$, $e_i^\top W'' v^* \sim \mathcal{N}(0,\|v\|^2/n) = \mathcal{N}(0,\ell/n)$ and therefore
\begin{equation}
\PP_n\left[|e_i^\top W'' v^*| \ge \mu s\right] \le 2 \exp\left(-\frac{n\mu^2 s^2}{2\ell}\right).
\end{equation}
By a union bound over all $i \in [n]$,
\begin{align*}
\PP_n\left[|e_i^\top W'' v^*| \le \mu s \text{ for all } i\right] &\ge 1 - 2n \exp\left(-\frac{n\mu^2 s^2}{2\ell}\right) \\
&\ge 1 - \exp\left(\log (2n) -\frac{n\mu^2 s^2}{2\ell}\right)\\
&\ge 1 - \exp\left(-\frac{n\mu^2 s^2}{4\ell}\right) \\
&= 1-\exp\left(-\frac{\bar{\lambda}^2}{144A^4}\frac{\ell}{\rho^2 n}\right)
\end{align*}
under the condition
$$\frac{n\mu^2 s^2}{4\ell}\ge \log(2n)\ \Leftarrow\ \rho \le \frac{\bar{\lambda}}{13A^2} \sqrt{\frac{\ell}{n \log n}} = \frac{\lambda}{13\sqrt{2}A^2} \sqrt{\frac{\ell}{n \log n}},$$
which, combined with \eqref{wig-cond}, is equivalent membership in the interval for $\ell$ that we are considering per \eqref{eq:wgn-rec-ell}.
Therefore, with probability at least
\[ 1-3\exp\left(-\frac{\bar{\lambda}^2}{64A^4}\frac{\ell^2}{\rho^2 n}\right)-\exp\left(-\frac{\bar{\lambda}^2}{144A^4}\frac{\ell}{\rho^2 n}\right) \ge 1-4\exp\left(-\frac{\lambda^2}{288A^4}\frac{\ell}{\rho^2 n}\right) \]
for all $j\in [n]$,
\[ j\in {\rm supp}(x) \hspace{1em}\text{ if and only if }\hspace{1em} |z_j|\ge \frac{s}{2} \]
and
\[ {\rm sign}(z_j) = {\rm sign}(x_j \la v^*,x\ra). \]
Thus, we find that thresholding the entries of $z$ at $s/2$ successfully recovers the support and signs of $x$, completing the proof.
\end{proof}

\begin{proof} [Proof of Theorem \ref{rec2-wgn} (Full Recovery).]
Since $Y_{\sI}\tilde{x} = \lambda_{\max}(Y_{\sI})\tilde{x}$, we must have ${\rm supp}(\tilde{x})\subset \sI$. Denote $W_{\sI} = P_{\sI}WP_{\sI}^\top$ and $\bar{W}_{\sI}$ the $\ell\times\ell$ submatrix of $W_{\sI}$ with rows and columns indexed by $\sI$ (the only nonzero rows and columns). Now, the variational description of the leading eigenvector yields
\begin{align*}
\tilde{x}^\top W_{\sI} \tilde{x}+\lambda\la \tilde{x},x\ra^2 = \tilde{x}^\top Y_{\sI}\tilde{x} \ge x^\top Y_{\sI} x
= x^\top W_{\sI}x +\lambda.
\end{align*}
Therefore,
$$\la \tilde{x},x\ra^2 \ge 1- \frac{1}{\lambda}(\tilde{x}^\top W_{\sI} \tilde{x}-x^\top W_{\sI} x)
\ge 1-\frac{1}{\lambda}(\lambda_{\max}(\bar{W}_{\sI})+\lambda_{\max}(-\bar{W}_{\sI})).$$
Note that $\bar{W}_{\sI}$ has the same law as $(\bar{G}+\bar{G}^\top) / \sqrt{2n}$, where $\bar{G}$ is an $\rho n\times\rho n$ matrix whose entries are independent standard normal random variables. Now, for any $\epsilon > \frac{4\sqrt{2\rho}}{\lambda}$, we have $\frac{\sqrt{2n}\lambda\epsilon}{4} > 2\sqrt{\rho n}$, a standard singular value estimate for Gaussian matrices (see \cite{vershynin-intro}, Corollary 5.35) gives
\begin{align*}
    \PP_n\left[\la \bar{x},x\ra^2 \le 1-\epsilon \right]
    &\le \Pr\left[\lambda_{\max}(\bar{W}_{\sI})+\lambda_{\max}(-\bar{W}_{\sI}) \ge \lambda\epsilon\right]\\
    &\le 2\Pr\left[\lambda_{\max}(\bar{W}_{\sI}) \ge \frac{\lambda\epsilon}{2}\right]\\
    &\le 2\Pr\left[\sigma_{\max}(\bar{G}) \ge \frac{\sqrt{2n}\lambda\epsilon}{4}\right]\\
    &\le 4\exp\left[-\frac{n}{16}\left(\lambda\epsilon-4\sqrt{2\rho}\right)^2\right],
\end{align*}
which conclues the proof.
\end{proof}

\section{Proofs for Low-Degree Likelihood Ratio Bounds}
\label{sec:proofs-ldlr}

\subsection{Low-Degree Likelihood Ratio for Spiked Models}

We begin by giving expressions for the norm of the low-degree likelihood ratio (LDLR) for the spiked Wigner and Wishart models. These expressions are derived in \cite{low-deg-notes} and \cite{BKW-sk}, respectively.

\begin{lemma}[$D$-LDLR for spiked Wigner model \cite{low-deg-notes}]
    Let $L_{n,\lambda,\sX}^{\le D}$ denote the degree-$D$ likelihood ratio for the spiked Wigner model with parameters $n,\lambda$ and spike prior $\sX$. Then,
    \begin{equation}
    \|L_{n,\lambda,\sX}^{\leq D}\|^2 = \Ex_{v^{(1)}, v^{(2)} \sim \sX_n}\left[\sum_{d = 0}^D \frac{1}{d!}\left(\frac{n}{2}\lambda^2\la v^{(1)}, v^{(2)} \ra^2\right)^d\right]
    \label{ldlr_wgn}
    \end{equation}
where $v^{(1)},v^{(2)}$ are drawn independently from $\sX_n$.
\end{lemma}

\begin{lemma}
[$D$-LDLR for spiked Wishart model \cite{BKW-sk}]
Let $L_{n,N,\beta,\sX}^{\le D}$ denote the degree-$D$ likelihood ratio for the spiked Wishart model with parameters $n,N,\beta$ and spike prior $\sX$. Define
\begin{align}
    \varphi_N(x) &\colonequals (1-4x)^{-N/2} \\
    \varphi_{N,k}(x) &\colonequals \sum_{d = 0}^k x^d\sum_{\substack{d_1,\dots,d_N\\ \sum d_i = d}}\prod_{i = 1}^N\binom{2d_i}{d_i},
\end{align}
so that $\varphi_{N,k}(x)$ is the Taylor series of $\varphi_N$ around $x = 0$ truncated to degree $k$. Then,
\begin{align}\label{ldlr_wsh}
\|L_{n,N,\beta,\sX}^{\le D}\|_{2}^2 &= \Ex_{v^{(1)},v^{(2)}\sim \sX_n}\left[\varphi_{N,\lfloor D/2 \rfloor}\left(\frac{\beta^2 \la v^{(1)},v^{(2)} \ra^2}{4}\right)\right]\nonumber\\
&= \Ex_{v^{(1)},v^{(2)}\sim \sX_n}\sum_{d = 0}^{\lfloor D/2 \rfloor}\left(\sum_{\substack{d_1,\dots,d_N\\ \sum d_i = d}}\prod_{i = 1}^N\binom{2d_i}{d_i}\right)\left(\frac{\beta^2\la v^{(1)},v^{(2)}\ra^2}{4}\right)^d,
\end{align}
where $v^{(1)},v^{(2)}$ are drawn independently from $\sX_n$.
\end{lemma}
We consider a signal $x$ drawn from the sparse Rademacher prior, $\sX_n = \sX_n^\rho$. The goal of this section is to prove upper and lower bounds on the LDLR expressions in \eqref{ldlr_wgn} and \eqref{ldlr_wsh} as $n\rightarrow\infty$, for certain regimes of the parameters ($\lambda,\rho$ for the Wigner model and $\beta,\gamma,\rho$ for the Wishart model).
These bounds are obtained in several steps. First, we treat the moment terms
\begin{equation}\label{def_A_d}
    A_d \colonequals (n\rho)^{2d}\Ex_{v^{(1)},v^{(2)}\sim \sX_n^\rho}\la v^{(1)},v^{(2)}\ra^{2d}
\end{equation}
from \eqref{ldlr_wgn} and \eqref{ldlr_wsh} in Section~\ref{subsec_A_d}, with upper bounds given in Lemmas~\ref{localubd} and~\ref{localubd2} and a lower bound given in Lemma~\ref{locallbd}. We then give a precise estimate in Lemma~\ref{coef} of the coefficient \[ \sum_{\substack{d_1,\dots,d_N\\ \sum d_i = d}}\prod_{i = 1}^N\binom{2d_i}{d_i} \]
in the LDLR \eqref{ldlr_wsh} of the Wishart model. Finally, by combining the above bounds, we show regimes of parameters under which the LDLR either remains bounded or diverges as $n \to \infty$. This yields the proofs of Theorems~\ref{lowbnd_wsh} and~\ref{upbnd_wsh} for the Wishart model, and Theorems~\ref{lowbnd_wgn} and~\ref{upbnd_wgn} for the Wigner model.

\subsection{Introduction and Estimates of \texorpdfstring{$A_d$}{A\_d}}\label{subsec_A_d}
In this section, we carry out combinatorial estimates of the moments $A_d$ defined in \eqref{def_A_d}, which appear in the LDLR expressions \eqref{ldlr_wgn} and \eqref{ldlr_wsh}. We give upper bounds (Lemmas \ref{localubd} and \ref{localubd2}) and a lower bound (Lemma \ref{locallbd}) on these moments.

For independent $v^{(1)},v^{(2)}$ drawn from the sparse Rademacher prior $\sX_n^\rho$, $\la v^{(1)},v^{(2)}\ra$ has the same distribution as $S_{n,\rho} = \frac{1}{n}\sum_{i = 1}^n s_{i,\rho}$ for i.i.d.\ $s_{i,\rho}$ with
\begin{equation}
s_{i,\rho} = \left\{
\begin{array}{cll}
+1 / \rho & \text{ with probability } & \rho^2 / 2, \\
-1 / \rho & \text{ with probability } & \rho^2 / 2, \\
0 & \text{ with probability } & 1 - \rho^2,
\end{array}
\right.
\end{equation}
and $k$th moment (for $k>0$) given by
\begin{equation}
    \EE_{s_{i, \rho}^k} = \left\{\begin{array}{cl} 0 & \text{ for } k \text{ odd}, \\ \rho^{2 - k} & \text{ for } k \text{ even}. \end{array}\right.
\end{equation}
Therefore, the moments of $\la v^{(1)},v^{(2)}\ra$ have the combinatorial description
\begin{align}\label{count}
\Ex_{v^{(1)},v^{(2)}\sim \sX_n^\rho}\la v^{(1)},v^{(2)}\ra^{2d} \nonumber &= n^{-2d}\mathbb{E}S_{n,\rho}^{2d} \\ 
&= n^{-2d}\sum_{i_1,\dots,i_{2d}\in [n]}\mathbb{E}s_{i_1,\rho}s_{i_2,\rho}\dots s_{i_{2d},\rho}\nonumber\\
&= n^{-2d}\sum_{\substack{a_1, \dots, a_n \geq 0 \\ \sum a_i = d}} \binom{2d}{2a_1 \hspace{0.2cm} \cdots \hspace{0.2cm} 2a_n}\ \mathbb{E}\prod_{a_j > 0}s_{j,\rho}^{2a_j}\nonumber\\
&= n^{-2d}\sum_{\substack{a_1, \dots, a_n \geq 0 \\ \sum a_i = d}} \binom{2d}{2a_1 \hspace{0.2cm} \cdots \hspace{0.2cm} 2a_n}\ \prod_{a_j > 0} \rho^{2-2a_j}\nonumber\\
&= n^{-2d}\sum_{\substack{a_1, \dots, a_n \geq 0 \\ \sum a_i = d}} \binom{2d}{2a_1 \hspace{0.2cm} \cdots \hspace{0.2cm} 2a_n}\ \rho^{2|\{ i:\ a_i > 0 \}|-2d}.
\end{align}

Recall, from \eqref{def_A_d}, that
\begin{equation*}
    A_d = (n\rho)^{2d}\Ex_{v^{(1)},v^{(2)}\sim \sX_n^\rho}\la v^{(1)},v^{(2)}\ra^{2d} = \sum_{\substack{a_1, \dots, a_n \geq 0 \\ \sum a_i = d}} \binom{2d}{2a_1 \hspace{0.2cm} \cdots \hspace{0.2cm} 2a_n}\ \rho^{2|\{i:\ a_i > 0\}|}.
\end{equation*}
Fix $d\le D_n$, and let $1\le k\le d$ be the number of positive numbers among the $\{a_i\}$. Suppose $d = w k+r$, where $0\le r< k$. For positive integers $b_1,b_2,\dots,b_k$ such that $\sum b_i = d$, we claim that
\begin{equation}\label{binom-estim}
\binom{2d}{2b_1 \hspace{0.2cm} \cdots \hspace{0.2cm} 2b_k} \le \binom{2d}{2w \hspace{0.2cm} \cdots \hspace{0.2cm} 2w \hspace{0.2cm} 2(w+1) \hspace{0.2cm} \cdots \hspace{0.2cm} 2(w+1)} \equalscolon M(k),
\end{equation}
and that equality holds if and only if $\{b_i\}$ consists of $r$ copies of $(w+1)$ and $k-r$ copies of $w$. This follows from the simple fact that, for any $1\le i,j\le k$ such that $b_i\ge b_j+2$, we have $(2b_i)!(2b_j)! > (2(b_i-1))!(2(b_j+1))!$, and therefore the left-hand side of the above inequality becomes strictly larger as we ``unbalance'' the $b_i$ by replacing $b_i$ and $b_j$ with $b_i -1$ and $b_j+1$. As a result, the left-hand side is maximized if and only if the maximum and minimum of $\{b_1,b_2,\dots,b_k\}$ differ by at most $1$. Now, since the total number of positive integer solutions to $\sum_{i = 1}^k b_i = d$ is $\binom{d-1}{k-1}$, we have
\begin{equation}\label{eq:AG}
A_d  \le \sum_{k = 1}^d \binom{n}{k}\binom{d-1}{k-1}\rho^{2k}M(k).
\end{equation}

Before proceeding to bounds on the $A_d$, we introduce the following result, which will be useful in several estimates in this section.
\begin{lemma} \label{approx-dn}
Suppose $D_n = o(n)$. Then, for sufficiently large $n$,
\[ \frac{n(n-1)\cdots (n-D_n+1)}{n^{D_n}} >\frac{1}{2} e^{-D_n^2/n}. \]
\end{lemma}
\begin{proof}
By Stirling's formula, as $n\rightarrow \infty$ and $n-D_n\rightarrow\infty$ (which is ensured by $D_n = o(n)$),
\begin{equation*}
\begin{aligned}
\frac{n(n-1)\cdots (n-D_n+1)}{n^{D_n}} = \frac{n!}{(n-D_n)! n^{D_n}}
&\sim \frac{\sqrt{2\pi n}(\frac{n}{e})^n}{\sqrt{2\pi (n-D_n)}(\frac{n-D_n}{e})^{n-D_n}n^{D_n}}\\
&\sim \frac{1}{e^{D_n}}\left(1+\frac{D_n}{n-D_n}\right)^{n-D_n}.
\end{aligned}
\end{equation*}
Here the relation $\sim$ means that the quotient of the quantities on either side tends to $1$ as $n\rightarrow\infty$. Since $D_n = o(n)$, for large enough $n$ such that $D_n < \frac{n}{3}$, we have
\begin{equation*}
\begin{aligned}
\log \left[\frac{1}{e^{D_n}}\left(1+\frac{D_n}{n-D_n}\right)^{n-D_n} \right] &= -D_n+(n-D_n)\log\left(1+\frac{D_n}{n-D_n}\right)\\
&\ge -D_n+(n-D_n)\left(\frac{D_n}{n-D_n} -\frac{D_n^2}{2(n-D_n)^2}\right)\\
&\ge -\frac{D_n^2}{n},
\end{aligned}
\end{equation*}
and the lemma follows.
\end{proof}

\begin{lemma} 
[First upper bound on $A_d$]\label{localubd}
In the setting of the spiked Wishart or Wigner model with sparse Rademacher prior $\sX_n^\rho$, suppose $D_n = o(n)$. If for some $\mu > 0$ we have
\begin{equation*}
\rho\ge \max\left(1,\sqrt{\frac{1}{6\mu}}\right)\sqrt{\frac{D_n}{n}},
\end{equation*}
then for sufficiently large $n$ and for any $1\le d\le D_n$, $A_d$ defined by \eqref{def_A_d} satisfies
\begin{equation}\label{eq:A-bound}
A_d \le \sum_{k = 1}^d G(k) \le 2d e^{\mu d+\frac{d^2}{n}} G(d) = 2d e^{\mu d+\frac{d^2}{n}}\binom{n}{d}\frac{(2d)!}{2^d}\rho^{2d},
\end{equation}
where
\[ G(k) \colonequals \binom{n}{k}\binom{d-1}{k-1}\rho^{2k}M(k). \]
\end{lemma}
\begin{proof}
Fix $1\le d\le D_n$. Recall that the first inequality in~\eqref{eq:A-bound} is a restatement of~\eqref{eq:AG}. By a simple comparison argument, we observe that $M(k)$ is monotone increasing with respect to $k$. For any $1\le k < \frac{d}{2}$, we have
$$\frac{G(k+1)}{G(k)} = \frac{\binom{n}{k+1}\binom{d-1}{k}\rho^{2k+2}M(k+1)}{\binom{n}{k}\binom{d-1}{k-1}\rho^{2k}M(k)} \ge \frac{(n-k)(d-k)}{k(k+1)}\rho^2 \ge \frac{(n-d)(\frac{d}{2})}{k(k+1)}\frac{d}{n} > 1$$
if $\rho \ge \sqrt{\frac{D_n}{n}}$. Therefore, $$\sum_{k <\frac{d}{2}}G(k) \le \left\lfloor\frac{d}{2}\right\rfloor\sum_{k\ge \frac{d}{2}}G(k).$$ Meanwhile, for $\frac{d}{2}\le k < d$, by \eqref{binom-estim},
$$M(k) = \max_{\substack{b_1, \dots, b_k > 0 \\ \sum b_i = d}}\binom{2d}{2b_1 \hspace{0.2cm} \cdots \hspace{0.2cm} 2b_k} = \binom{2d}{2 \hspace{0.2cm} \cdots \hspace{0.2cm} 2 \hspace{0.2cm} 4 \hspace{0.2cm} \cdots \hspace{0.2cm} 4} = \frac{(2d)!}{24^{d-k}2^{2k-d}},$$
since the maximum is attained when $\{b_i\}$ has $(d-k)$ occurrences of 2 and $(2k-d)$ occurrences of 1. As a result, with the help of Lemma \ref{approx-dn},
\begin{equation*}
\begin{aligned}
\frac{G(k)}{G(d)} &= \frac{\binom{n}{k}\binom{d-1}{k-1}\rho^{2k}M(k)}{\binom{n}{d}\rho^{2d}M(d)} \\
&\le \frac{1}{(6\rho^2)^{d-k}}\cdot\frac{\binom{n}{k}\binom{d}{k}}{\binom{n}{d}}\\
&\le \frac{1}{(6\rho^2)^{d-k}}\cdot \frac{d!}{(n-k)(n-k-1)\dots (n-d+1)k!}\cdot \frac{d!}{k!(d-k)!}\\
&\le \frac{1}{(6\rho^2)^{d-k}}\cdot \frac{2e^{d^2/n}}{n^{d-k}}\cdot \left(\frac{d!}{k!}\right)^2\frac{1}{(d-k)!}\\
&\le 2e^{d^2/n}\left(\frac{d^2}{6\rho^2 n}\right)^{d-k}\cdot \frac{1}{(d-k)!}.
\end{aligned}
\end{equation*}
Thus for $\rho \ge \sqrt{\frac{D}{6\mu n}} \ge \sqrt{\frac{d}{6\mu n}}$
we have $\frac{d^2}{6\rho^2 n} \le \mu d$ and
\[ \sum_{k = 1}^d G(k) \le \left(\left\lfloor d/2\right\rfloor+1\right)\sum_{k\ge \frac{d}{2}}G(k) \le 2de^{d^2/n}\sum_{k\ge \frac{d}{2}}\frac{(\mu d)^{d-k}}{(d-k)!}G(d)\le 2d e^{\mu d+d^2/n} G(d), \]
completing the proof.
\end{proof}

\begin{lemma}
[Second upper bound on $A_d$]\label{localubd2}
In the setting of the spiked Wishart or Wigner model with sparse Rademacher prior $\sX_n^\rho$, suppose $D_n = o(n).$ If for some $\mu < 1/\sqrt{3}$ we have
$$\rho \ge \mu\sqrt{\frac{D_n}{n}},$$
then for sufficiently large $n$ and for any $11\le d\le D_n$, 
$$A_d \le \sum_{k = 1}^d G(k) \lesssim \sqrt{d}e^{d^2/n}\left(\frac{11e}{30}\right)^{d/2}\mu^{-2d} G(d) = \sqrt{d}e^{d^2/n}\left(\frac{11e}{30}\right)^{d/2}\mu^{-2d}\binom{n}{d}\frac{(2d)!}{2^d}\rho^{2d},$$
where $A_d$ is defined in \eqref{def_A_d} and $G(k)$ is defined as in Lemma \ref{localubd}.
\end{lemma}
\begin{proof}
As in the proof of Lemma \ref{localubd}, for sufficiently large $n$ and for any $1\le d\le D_n$ and any $1\le k <\frac{d}{2}$, we have
$$\frac{G(k+1)}{G(k)} \ge \frac{(n-k)(d-k)}{k(k+1)}\rho^2 \ge \frac{(n-k)(d-k)}{k(k+1)}\left(\mu^2\frac{d}{n}\right) \ge \mu^2.$$
Therefore,
$$\sum_{k < \frac{d}{2}}G(k) \le (\mu^{-2}+\mu^{-4}+\dots+\mu^{-2(\lceil d/2 \rceil -1)})G(\lceil d/2 \rceil) \lesssim \mu^{-2\lceil d/2 \rceil}\sum_{k\ge \lceil d/2 \rceil}G(k).$$
Meanwhile, for $\frac{d}{2}\le k < d$,
$$\frac{G(k)}{G(d)} \le 2e^{d^2/n}\left(\frac{d^2}{6\rho^2 n}\right)^{d-k}\cdot \frac{1}{(d-k)!} \le 2e^{d^2/n}\left(\frac{d}{6\mu^2}\right)^{d-k}\cdot\frac{1}{(d-k)!}.$$
Summing these quantities, we find (the last inequality requiring $d\ge 11$)
\begin{align*}
\frac{\sum_{k\ge d/2}G(k)}{G(d)}
&\le 2e^{d^2/n}\sum_{k \le d/2}\left(\frac{d}{6\mu^2}\right)^{k}\cdot\frac{1}{k!}\\
&\lesssim de^{d^2/n}\left(\frac{d}{6\mu^2}\right)^{\lfloor d/2 \rfloor}\cdot \frac{1}{\lfloor d/2 \rfloor!}\\
&\lesssim \sqrt{d}e^{d^2/n}\left(\frac{ed}{6\mu^2\lfloor d/2 \rfloor}\right)^{\lfloor d/2 \rfloor}\\
&\lesssim \sqrt{d}e^{d^2/n}\left(\frac{11e}{30}\right)^{\lfloor d/2 \rfloor}\mu^{-2\lfloor d/2 \rfloor}.
\end{align*}
Here we have used the fact that $(d / 6\mu^2)^{k} / k!$ is monotone increasing for $1\le k \le \frac{d}{2}$, since $d / 6\mu^2 > d / 2$. Combining the two cases, we conclude that
\[ \frac{\sum_{k = 1}^dG(k)}{G(d)} \lesssim \frac{(1+\mu^{-2\lceil d/2 \rceil})\sum_{k \ge d/2}G(k)}{G(d)} \lesssim \sqrt{d}e^{d^2/n}\left(\frac{11e}{30}\right)^{d/2}\mu^{-2d}, \]
completing the proof.
\end{proof}

\begin{lemma}
[Lower bound on $A_d$]\label{locallbd}
In the settings of the spiked Wishart or Wigner models with sparse Rademacher prior $\sX_n^\rho$, consider a series $d = d_n = o({n})$ with integers $w = w_n$ satisfying $w_n \mid d_n$. Then, as $d/w\rightarrow\infty$, $A_d$ defined by \eqref{def_A_d} satisfies
$$A_d \ge (1-o(1))\binom{n}{d}\frac{(2d)!\sqrt{w}}{2^d}\ \left[2\left(\frac{d}{ne\rho^2}\right)^{1-\frac{1}{w}}\left(\frac{w}{(2w)!}\right)^{\frac{1}{w}}\right]^d\rho^{2d}.$$
\end{lemma}
\begin{proof}
To obtain a lower bound on $A_d$, we only consider the contribution to the sum from terms $\{a_i\}$ with $\frac{d}{w}$-many occurrences of $w$ and $(n-\frac{d}{w})$-many occurrences of zero:
\[ \rho^{-2d}A_d = \sum_{\substack{a_1, \dots, a_n \geq 0 \\ \sum a_i = d}} \binom{2d}{2a_1 \hspace{0.2cm} \cdots \hspace{0.2cm} 2a_n}\ \rho^{2|\{i:\ a_i > 0\}|-2d} \ge \binom{n}{d/w}\frac{(2d)!}{[(2w)!]^{d/w}}\rho^{\frac{2d}{w}-2d} \equalscolon T_{n,d}(w). \]
Now, we calculate the ratio
\begin{align}\label{calc}
\frac{T_{n,d}(w)}{T_{n,d}(1)} &= \frac{d!(n-d)!}{(\frac{d}{w})!(n-\frac{d}{w})!}\frac{2^d}{[(2w)!]^{\frac{d}{w}}}\rho^{\frac{2d}{w}-2d}\nonumber\\
&= \frac{d!/(\frac{d}{w})!}{(n-\frac{d}{w})\left(n-\frac{d}{w}-1\right)\cdots (n-d+1)}\left[\frac{2}{[(2w)!]^{\frac{1}{w}}\rho^{2(1-\frac{1}{w})}}\right]^d.
\end{align}
By Stirling's formula, for $w$ fixed and $d$ sufficiently large,
\[ \frac{d!}{(d/w)!} = (1+o(1))\sqrt{w}(d/e)^{d-\frac{d}{w}}w^{\frac{d}{w}}. \]
Meanwhile,
\[ \left(n-\frac{d}{w}\right)\left(n-\frac{d}{w}-1\right)\cdots (n-d+1) \le n^{d-\frac{d}{w}}. \]
Plugging into \eqref{calc} we get
\begin{equation*}
\frac{T_{n,d}(w)}{T_{n,d}(1)} \ge (1-o(1))\sqrt{w}\ \left[2\left(\frac{d}{ne\rho^2}\right)^{1-\frac{1}{w}}\left(\frac{w}{(2w)!}\right)^{\frac{1}{w}}\right]^d,
\end{equation*}
completing the proof.
\end{proof}

\subsection{The Wishart Model}
In this section, we first carry out an estimate on the extra coefficient occurring in the Wishart LDLR \eqref{ldlr_wsh} (Lemma \ref{coef}), then use the bounds on $A_d$ (Lemmas~\ref{localubd}, \ref{localubd2} and~\ref{locallbd}) to prove the upper bound (Theorem~\ref{lowbnd_wsh}) and the lower bound (Theorem~\ref{upbnd_wsh}) on \eqref{ldlr_wsh}.
\begin{lemma}
[Bounds on coefficient in Wishart LDLR]\label{coef}
Suppose $D_n = o(N)$. There exist absolute constants $c_1,c_2 > 0$ such that, for sufficiently large $N$, for any $1\le d\le D_n$,
\begin{equation}
\frac{(2N)^d}{d!} \le \sum_{\substack{d_1,\dots,d_N\\ \sum d_i = d}}\prod_{i = 1}^N\binom{2d_i}{d_i} \le c_1 d^{3/2}e^{c_2d^2/N}\frac{(2N)^d}{d!}.
\end{equation}
\end{lemma}
\begin{proof}
For the lower bound, note that for any $d_i \ge 2$,
\[ \binom{2d_i}{d_i} = \frac{2d_i(2d_i-1)}{d_i^2} \binom{2(d_i-1)}{d_i-1} \ge \binom{2}{1}\binom{2(d_i-1)}{d_i-1}, \]
so for any $d_1,\dots,d_N \geq 0$ such that $\sum d_i = d$,
\[ \prod_{i = 1}^N\binom{2d_i}{d_i} \ge 2^d. \]
Summing over all of the $\{d_i\}$, we find
\[ \sum_{\substack{d_1,\dots,d_N\\ \sum d_i = d}}\prod_{i = 1}^N\binom{2d_i}{d_i} \ge \binom{N+d-1}{d}2^d \ge \frac{(2N)^d}{d!}. \]
For the upper bound, we separately consider those $\{d_i\}$'s with exactly $k$ positive entries for each $k = 1,2,\dots,d$:
\begin{equation*}
\begin{aligned}
\sum_{\substack{d_1,\dots,d_N\\ \sum d_i = d}}\prod_{i = 1}^N\binom{2d_i}{d_i} &= \sum_{k = 1}^d \binom{N}{k}\sum_{\substack{c_1,\dots,c_k > 0\\ \sum c_i = d}}\prod_{i = 1}^k\binom{2c_i}{c_i}\\
& \le\sum_{k = 1}^d \binom{N}{k}\binom{d-1}{k-1} \max_{\substack{c_1,\dots,c_k > 0\\ \sum c_i = d}}\prod_{i = 1}^k\binom{2c_i}{c_i}.
\end{aligned}
\end{equation*}
Given any positive integers $c_1,\dots,c_k$ such that $\sum c_i = d$, if there are two entries $c_j\ge c_\ell\ge 2$, consider $\tilde{c}_j = c_j+1$, $\tilde{c}_\ell = c_\ell-1$ and $\tilde{c}_i = c_i$ for all $i\neq j,\ell$. We have the comparison
\[ \frac{\prod_{i = 1}^k\binom{2\tilde{c}_i}{\tilde{c}_i}}{\prod_{i = 1}^k\binom{2c_i}{c_i}} = \frac{4-\frac{2}{c_j+1}}{4-\frac{2}{c_l+1}} > 1. \]
Therefore, for fixed $k$, the product $\prod_{i = 1}^k\binom{2c_i}{c_i}$ is maximized when $\{c_i\}$ is composed of $(k-1)$ occurrences of 1 and one occurrence of $(d-k+1)$. As a result,
\[ \sum_{\substack{d_1,\dots,d_N\\ \sum d_i = d}}\prod_{i = 1}^N\binom{2d_i}{d_i} \le \sum_{k = 1}^d \binom{N}{k}\binom{d-1}{k-1} \cdot 2^{k-1}\binom{2(d-k+1)}{d-k+1} \equalscolon \sum_{k = 1}^d S(k). \]
Since
\begin{equation}\label{rep-s}
S(k) = \binom{N}{k}\binom{d-1}{k-1} \cdot 2^{k-1}\binom{2(d-k+1)}{d-k+1} \le \frac{N^k}{k!}\binom{d}{k}\cdot 2^k\binom{2(d-k)}{d-k}
\end{equation}
and Stirling's formula gives
\begin{align*}
    k! &\gtrsim \left(\frac{k}{e}\right)^k,\\
    \binom{d}{k} &\lesssim \frac{d^d}{k^k (d-k)^{d-k}}, \\
\binom{2(d-k)}{d-k} &\lesssim 4^{d-k}, \\ d! &\lesssim \sqrt{d}\left(\frac{d}{e}\right)^d,
\end{align*}
substituting into \eqref{rep-s} and denoting $k = (1-\eta)d$, we find
\begin{equation*}
\begin{aligned}
S(k) \left(\frac{(2N)^d}{d!}\right)^{-1} &\lesssim \left(\frac{Ne}{k}\right)^k\cdot \frac{d^d}{k^k (d-k)^{d-k}} \cdot 2^{2d-k} \cdot \left(\frac{1}{\sqrt{d}}\left(\frac{2Ne}{d}\right)^d\right)^{-1}\\
& = \sqrt{d}\left(\frac{2d}{Ne\eta}\right)^{\eta d} \frac{1}{(1-\eta)^{2(1-\eta)d}}.
\end{aligned}
\end{equation*}
Now, for $\eta \in (0,1]$, denote 
\[ h(\eta) \colonequals \left(\frac{2d}{Ne\eta}\right)^{\eta} \frac{1}{(1-\eta)^{2(1-\eta)}}. \]
Then,
\begin{equation}\label{estim-up-wish}
\sum_{\substack{d_1,\dots,d_N\\ \sum d_i = d}}\prod_{i = 1}^N\binom{2d_i}{d_i} \le \sum_{k = 1}^d S(k) \le d\max_{1\le k\le d}S(k) \lesssim d^{3/2} \frac{(2N)^d}{d!}\left[\sup_{\eta\in (0,1]}h(\eta)\right]^d.
\end{equation}
The last step is to evaluate $h(\eta)$. Note that
\[ \frac{d}{d \eta} [\log h(\eta)] = \log\left(\frac{2d}{N}\right)-\log(\eta)+2\log(1-\eta). \]
Since $\frac{2d}{N}\le \frac{2D_n}{N} = o(1)$ as $n\rightarrow\infty$, for large $N$ the unique maximizer of $h$ has the form
\[ \eta^* = \eta^*(N,d) = \frac{(2-o(1))d}{N}, \]
and consequently
\[ \sup_{\eta\in(0,1]}h(\eta) = h(\eta^*) = \left(\frac{2}{e(2-o(1))}\right)^{\eta^*}(1-\eta^*)^{-2(1-\eta^*)}\le e^{4\eta^* (1-\eta^*)}\le e^{c_2d/N}. \]
Substituting into \eqref{estim-up-wish} then completes the proof.
\end{proof}

\begin{proof}[Proof of Theorem \ref{lowbnd_wsh}(a).] Suppose \eqref{lowbnd1'} holds. Let $\mu = -\log\hat{\lambda}$. In the setting of Lemma~\ref{coef}, note that
\begin{align*}
    \lim_{n\rightarrow\infty}\frac{D_n}{N} &= 0, \\
\liminf_{n\rightarrow\infty}\left(-\frac{1}{2}\log\hat{\lambda}_n\right) = -\frac{1}{2}\log\left(\limsup_{n\rightarrow\infty}\hat{\lambda}_n\right) &> 0.
\end{align*}
For $n$ large enough so that $(c_2+\frac{1}{\gamma})\frac{D_n}{N} < -\frac{1}{2}\log\hat{\lambda}_n$, applying Lemma~\ref{localubd} in the expression of \eqref{ldlr_wsh} yields
\begin{equation*}
\begin{aligned}
\|L_{n,N,\beta,\sX}^{\le D}\|_{2}^2 &= \Ex_{v^{(1)},v^{(2)}\sim\sX_n^\rho}\sum_{d = 0}^{\lfloor D/2 \rfloor}\left[\sum_{\substack{d_1,\dots,d_N\\ \sum d_i = d}}\prod_{i = 1}^N\binom{2d_i}{d_i}\right]\left[\frac{\beta^2\la v^{(1)},v^{(2)}\ra^2}{4}\right]^d\\
&\lesssim \sum_{d = 1}^{\lfloor D/2 \rfloor} \left[d^{3/2}\frac{(2N)^d}{d!} e^{c_2 d^2/N}\right]\left[\left(\frac{\beta^2}{4}\right)^d (n\rho)^{-2d}\cdot 2d e^{\mu d+d^2/n}\binom{n}{d}\frac{(2d)!}{2^d}\rho^{2d}\right]\\
&\le \sum_{d = 1}^{\lfloor D/2 \rfloor} 2d^{5/2}e^{c_2 dD/N}\frac{(2d)!}{4^d(d!)^2} \left[e^{\mu+D/n} \left(\frac{N}{n}\right) \beta^2\right]^d\\
&\lesssim \sum_{d = 1}^{\infty} d^{2}\left(e^{(c_2+\frac{1}{\gamma})\frac{D}{N}}e^\mu  {\hat{\lambda}}^2\right)^d \\
&\le \sum_{d = 1}^{\infty} d^2({\hat{\lambda}}^{1/2})^d \\
&= O(1),
\end{aligned}
\end{equation*}
where the last equality is by the assumption that $\limsup_{n\rightarrow\infty}\hat{\lambda}_n < 1$.
\end{proof}

\begin{proof}[Proof of Theorem \ref{lowbnd_wsh}(b).] If \eqref{lowbnd2'} holds, then $\hat{\lambda}_n\le 1/\sqrt{3}$ for sufficiently large $n$. In the setting of Lemma \ref{coef}, suppose $(c_2+\frac{1}{\gamma})\frac{D}{N} < 0.001$. Then, substituting the estimates in Lemma~\ref{localubd2} (taking $\mu = \hat{\lambda}$) into \eqref{ldlr_wsh} gives
\begin{equation*}
\begin{aligned}
\|L_{n,N,\beta,\sX}^{\le D}\|_{2}^2 &\lesssim \sum_{d = 11}^{\lfloor D/2 \rfloor} \left[d^{3/2}\frac{(2N)^d}{d!} e^{c_2d^2/N}\right]\left[\left(\frac{\beta^2}{4}\right)^d (n\rho)^{-2d}\cdot \sqrt{d}e^{d^2/n}\left(\frac{11e}{30}\right)^{\frac{d}{2}}\hat{\lambda}^{-2d}\binom{n}{d}\frac{(2d)!}{2^d}\rho^{2d}\right]\\
&\lesssim \sum_{d = 11}^{\lfloor D/2 \rfloor} d^{2}e^{(c_2+\frac{1}{\gamma})\frac{dD}{N}}\frac{(2d)!}{4^d(d!)^2}\left[\beta^2\left(\frac{N}{n}\right)\hat{\lambda}^{-2}\right]^d \left(\frac{11e}{30}\right)^{d/2}\\
&\lesssim \sum_{d = 11}^{\infty} d^{3/2}\left(e^{0.001}\sqrt{\frac{11e}{30}}\right)^d \\
&= O(1),
\end{aligned}
\end{equation*}
completing the proof.
\end{proof}

\begin{proof}[Proof of Theorem \ref{upbnd_wsh}(a).]
In \eqref{count}, only counting the terms $\{a_i\}$ with $d$ occurrences of 1 and $(n-d)$ occurrences of zero yields
\begin{equation}\label{triviallbd}
\Ex_{v^{(1)},v^{(2)}\sim \sX_n^\rho}\la v^{(1)},v^{(2)}\ra^{2d} \ge n^{-2d}\binom{n}{d}\frac{(2d)!}{2^d}.
\end{equation}
From Lemma \ref{approx-dn} we have that when $D_n = o(n)$, for sufficiently large $n$ and for any $1\le d\le D_n$,
\begin{equation}\label{lowbnd-binom}
    \binom{n}{d} = \frac{n(n-1)\dots(n-d+1)}{n^d}\cdot \frac{n^d}{d!} \ge \frac{1}{2}e^{-d^2/n}\frac{n^d}{d!}\ge \frac{1}{2}e^{-dD_n/n}\frac{n^d}{d!}.
\end{equation}
Substituting Lemma~\ref{coef}, \eqref{triviallbd} and \eqref{lowbnd-binom} into \eqref{ldlr_wsh} yields, for sufficiently large $n$,
\begin{align*}
    \|L_{n,N,\beta,\sX}^{\le D}\|_{2}^2 &\ge \sum_{d = 1}^{D_n} \left[\frac{(2N)^d}{d!}\right]\left[\frac{\beta^{2d}}{4^d} n^{-2d}\binom{n}{d}\frac{(2d)!}{2^d}\right]\\
    &\gtrsim \sum_{d = 1}^{D_n} \frac{(2d)!}{4^d(d!)^2}\left(\frac{\beta^2}{\gamma}\right)^d e^{-dD_n/n}\\
    &\gtrsim \sum_{d = 1}^{D_n} \frac{1}{\sqrt{d}}\left({\hat{\lambda}_n}^2 e^{-D_n/n}\right)^{d} \\ 
    &\ge \sum_{d = 1}^{D_n} \frac{1}{\sqrt{d}} \\
    &= \omega(1),
\end{align*}
since $D_n = \omega(1)$, $\liminf_{n\rightarrow\infty}\hat{\lambda}_n > 1$ and $e^{-D_n/n}\rightarrow 1$.
\end{proof}

\begin{lemma}\label{spsthm2}
Suppose $\omega(1) \le D_n \le o(n)$. If there exists a series of positive integers $w_n = o(\sqrt{D_n})$ such that
\begin{equation}\label{ineq2}
\liminf_{n\rightarrow\infty}\ 2\hat{\lambda}_n^2 \left(\frac{D_n}{2ne\rho_n^2}\right)^{1-\frac{1}{w_n}}\left(\frac{w_n}{(2w_n)!}\right)^{\frac{1}{w_n}} > 1,
\end{equation}
then $\|L_{n,N,\beta,\sX}^{\le D_n}\|_{2}^2 \to \infty$ as $n\rightarrow\infty$.
\end{lemma}

\begin{proof}
 If \eqref{ineq2} holds, we can choose an $\epsilon > 0$ such that, for sufficiently large $n$, 
$$2\hat{\lambda}_n^2 \left(\frac{D_n}{2ne\rho_n^2}\right)^{1-\frac{1}{w_n}}\left(\frac{w_n}{(2w_n)!}\right)^{\frac{1}{w_n}} > 1+\epsilon.$$
Let $n$ satisfy the above inequality. Pick $\mu\in (0,1)$ such that
\[ \mu^{1-\frac{1}{w_n}}(1+\epsilon) > 1, \]
which implies
\begin{equation}\label{ineq3}
2 \left(\frac{\mu D_n}{2ne\rho_n^2}\right)^{1-\frac{1}{w_n}}\left(\frac{w_n}{(2w_n)!}\right)^{\frac{1}{w_n}} > \frac{\gamma}{\beta^2}.
\end{equation}
In the sum \eqref{ldlr_wsh}, we only consider those $d \in (\mu D_n/2,\lfloor D_n/2 \rfloor)$ that are multiples of $w_n$. By Lemma~\ref{coef}, Lemma~\ref{locallbd}, and \eqref{ineq3},
\begin{equation*}
\begin{aligned}
&\Ex_{v^{(1)}, v^{(2)} \sim \sX_n^\rho}\left(\sum_{\substack{d_1,\dots,d_N\\ \sum d_i = d}}\prod_{i = 1}^N\binom{2d_i}{d_i}\right)\left(\frac{\beta^2\la v^{(1)},v^{(2)}\ra^2}{4}\right)^d \\
&\hspace{2cm}\gtrsim \frac{(2N)^d}{d!}\left[\left(\frac{\beta^2}{4}\right)^d(n\rho)^{-2d}\binom{n}{d}\frac{(2d)!\sqrt{w_n}}{2^d}\rho^{2d}\cdot \left(\frac{\gamma}{\beta^2}\right)^d\right]\\
&\hspace{2cm}\gtrsim \frac{(2d)!}{4^d(d!)^2} \left(\frac{N\gamma}{n}\right)^d \\
&\hspace{2cm}\gtrsim \frac{1}{\sqrt{d}}.
\end{aligned}
\end{equation*}
Therefore
\begin{align*}
\|L_{n,N,\beta,\sX}^{\le D_n}\|_{2}^2 &\gtrsim \sum_{\substack{\mu D_n/2 < d < \lfloor D_n/2 \rfloor\\ w_n\ |\ d}}\frac{1}{\sqrt{d}} \\ &\gtrsim \frac{1}{\sqrt{w_n}}\left(\sqrt{\frac{ D_n}{2w_n}}-\sqrt{\frac{\mu D_n}{2w_n}}\right)\\
&= \frac{1-\sqrt{\mu}}{\sqrt{2}}\frac{\sqrt{D_n}}{w_n} \\ &= \omega(1),
\end{align*}
completing the proof.
\end{proof}

\begin{proof}[Proof of Theorem \ref{upbnd_wsh}(b).]
For sufficiently large $n$, in Lemma \ref{spsthm2} we choose
\[ w_n = \lceil \log(1/\hat{\lambda}_n) \rceil = \lceil \log(\sqrt{\gamma}/\beta) \rceil\, \]
which is $o(\sqrt{D_n})$. Reorganizing the terms, the condition \eqref{ineq2} is satisfied if for sufficiently large $n$,
\begin{equation}\label{suf-rho-wsh}
\rho_n < 0.99\frac{1}{\sqrt{2e}}\left(\frac{w_n\cdot 2^{w_n}}{(2w_n)!}\right)^{1/(w_n-1)} \sqrt{\frac{D_n}{n}}\hat{\lambda}_n^{w_n/(w_n-1)}.
\end{equation}
Notice that
\begin{align*}
    \frac{1}{\sqrt{2e}}\left(\frac{w_n\cdot 2^{w_n}}{(2w_n)!}\right)^{1/(w_n-1)} &= \Theta(w_n^{-2}) \\ &= \Theta(\log^{-2} (1/\hat{\lambda}_n)), \\ \hat{\lambda}_n^{w_n/(w_n-1)} &= \hat{\lambda}_n\cdot \hat{\lambda}_n^{1/(\lceil \log(1/\hat{\lambda}_n) \rceil-1)} \\ &= \Theta(\hat{\lambda}_n).
\end{align*}
Therefore, there exists an absolute constant $C$ such that, if
\[ \rho_n < C\sqrt{\frac{D_n}{n}}\hat{\lambda}_n\log^{-2}(1/\hat{\lambda}_n), \]
then \eqref{suf-rho-wsh} is satisfied and the divergence of $\|L_{n,N,\beta,\sX}^{\le D_n}\|_{2}^2$ follows from Lemma \ref{spsthm2}.
\end{proof}

\subsection{The Wigner Model}
In this section, we use the bounds on $A_d$ (Lemmas~\ref{localubd}, \ref{localubd2} and~\ref{locallbd}) to prove the upper bound (Theorem~\ref{lowbnd_wgn}) and the lower bound (Theorem~\ref{upbnd_wgn}) on the Wigner LDLR \eqref{ldlr_wgn}.
\begin{proof}[Proof of Theorem \ref{lowbnd_wgn}(a).]  We only work with those $n$ for which \eqref{lowbnd1} holds. Let $\mu = -\log\lambda$. Note that
\begin{align*}
    \lim_{n\rightarrow\infty}\frac{D_n}{N} &= 0, \\
\liminf_{n\rightarrow\infty}\left(-\frac{1}{2}\log\hat{\lambda}_n\right) = -\frac{1}{2}\log\left(\limsup_{n\rightarrow\infty}\hat{\lambda}_n\right) &> 0.
\end{align*}
For large enough $n$ that $\frac{D_n}{n} < -\frac{1}{2}\log\lambda_n$, applying Lemma \ref{localubd} in the expression of \eqref{ldlr_wgn} yields
\begin{equation*}
\begin{aligned}
\|L_{n,\lambda,\sX}^{\le D_n}\|^2 &= \Ex_{v^{(1)},v^{(2)}\sim\sX_n^\rho}\left[\sum_{d = 0}^{D_n} \frac{1}{d!}\left(\frac{n}{2}\lambda^2\right)^d \la v^{(1)},v^{(2)} \ra^{2d}\right]\\
&\lesssim \sum_{d = 1}^{D_n}\frac{1}{d!}\left(\frac{n}{2}\lambda^2\right)^d(n\rho)^{-2d} \cdot 2d e^{\mu d+d^2/n}\binom{n}{d}\frac{(2d)!}{2^d}\rho^{2d}\\
&\lesssim \sum_{d = 1}^{D_n}\frac{d(2d)!}{4^d(d!)^2}(e^{\mu+D_n/n} \lambda^2)^d\\
&\lesssim \sum_{d = 1}^{D_n}\frac{1}{\sqrt{d}}(\lambda^{1/2})^d \\ &= O(1),
\end{aligned}
\end{equation*}
where the last equation is by the assumption $\limsup_{n\rightarrow\infty}\lambda_n < 1$.
\end{proof}

\begin{proof}[Proof of Theorem \ref{lowbnd_wgn}(b).] By the assumption \eqref{lowbnd2}, for sufficiently large $n$ we have $\lambda_n\le 1/\sqrt{3}$. Now, Theorem \ref{lowbnd_wgn}(b) immediately follows from Lemma \ref{localubd2} (taking $\mu = \lambda$), since for large enough $n$ that $\frac{D_n}{n} < 0.001$, we have
\begin{equation*}
\begin{aligned}
\|L_{n,\lambda,\sX}^{\le D_n}\|^2 
&\lesssim \sum_{d = 11}^{D_n}\frac{1}{d!}\left(\frac{n}{2}\lambda^2\right)^d(n\rho)^{-2d} \cdot \sqrt{d}e^{d^2/n}\left(\frac{11e}{30}\right)^{d/2}\lambda^{-2d}\binom{n}{d}\frac{(2d)!}{2^d}\rho^{2d}\\
&\lesssim \sum_{d = 11}^{\infty}\frac{\sqrt{d}(2d)!}{4^d(d!)^2}e^{d^2/n}\left(\frac{11e}{30}\right)^{d/2}\\
&\lesssim \sum_{d = 11}^{\infty}\left(e^{D_n/n}\sqrt{\frac{11e}{30}}\right)^d \\
&\lesssim \sum_{d = 11}^{\infty}\left(e^{0.001}\sqrt{\frac{11e}{30}}\right)^d \\
&=O(1),
\end{aligned}
\end{equation*}
completing the proof.
\end{proof}

\begin{proof}[Proof of Theorem \ref{upbnd_wgn}(a).]
Substituting \eqref{triviallbd} and \eqref{lowbnd-binom} into \eqref{ldlr_wgn} yields
\begin{align*}
    \|L_{n,\lambda,\sX}^{\le D_n}\|_2^2 &\ge 
    \sum_{d = 1}^{D_n} \frac{1}{d!}\left(\frac{n}{2}\lambda^2\right)^d\cdot n^{-2d}\binom{n}{d} \frac{(2d)!}{2^d}\\
    &\gtrsim \sum_{d = 1}^{D_n} \frac{(2d)!}{4^d(d!)^2}\lambda^{2d}e^{-dD_n/n}\\ &\gtrsim \sum_{d = 1}^{D_n} \frac{1}{\sqrt{d}}\left(\lambda^2 e^{-D_n/n}\right)^d \\ &\ge \sum_{d = 1}^{D_n} \frac{1}{\sqrt{d}} \\ &= \omega(1),
\end{align*}
since $D_n = \omega(1)$, $\liminf_{n\rightarrow\infty}\lambda_n > 1$ and $e^{-D_n/n}\rightarrow 1$.
\end{proof}

\begin{lemma}\label{spsthm}
Suppose $\omega(1) \le D_n \le o(n)$. If there exists a series of positive integers $w_n = o(\sqrt{D_n})$ such that
\begin{equation}\label{ineq}
\liminf_{n\rightarrow\infty}\ 2{\lambda}_n^2 \left(\frac{D_n}{ne\rho_n^2}\right)^{1-\frac{1}{w_n}}\left(\frac{w_n}{(2w_n)!}\right)^{\frac{1}{w_n}} > 1
\end{equation}
then $\|L_{n,\lambda,\sX}^{\le D_n}\|_2^2 \to \infty$ as $n\rightarrow\infty$.
\end{lemma}

\begin{proof}
If \eqref{ineq} holds, we can choose an $\epsilon > 0$ such that for sufficiently large $n$,
\[ 2\lambda_n^2 \left(\frac{D_n}{ne\rho_n^2}\right)^{1-\frac{1}{w_n}}\left(\frac{w_n}{(2w_n)!}\right)^{\frac{1}{w_n}} > 1+\epsilon. \]
Let $n$ satisfy the above inequality. Pick $\mu\in (0,1)$ such that
\[ \mu^{1-\frac{1}{w_n}}(1+\epsilon) > 1. \]
In the sum \eqref{ldlr_wgn} we only consider those $d > \mu D_n$ that are multiples of $w_n$. For each of them, Lemma~\ref{locallbd} gives
\begin{equation*}
\begin{aligned}
\frac{1}{d!}\left(\frac{n}{2}\lambda_n^2\right)^d \mathbb{E}\la v^{(1)},v^{(2)} \ra^{2d} &\gtrsim \frac{1}{d!}\left(\frac{n}{2}\lambda_n^2\right)^d\cdot n^{-2d}\binom{n}{d}\frac{(2d)!}{2^d}\ \left[2\left(\frac{d}{ne\rho^2}\right)^{1-\frac{1}{w_n}}\left(\frac{w_n}{(2w_n)!}\right)^{\frac{1}{w}}\right]^d\\
&\ge \frac{(2d)!}{4^d(d!)^2}\cdot \frac{n(n-1)\cdots (n-d+1)}{n^d} \cdot \left[2\lambda^2\left(\frac{\mu D_n}{ne\rho^2}\right)^{1-\frac{1}{w_n}}\left(\frac{w_n}{(2w_n)!}\right)^{\frac{1}{w_n}}\right]^d\\
&\gtrsim \frac{1}{\sqrt{d}}.
\end{aligned}
\end{equation*}
Therefore,
\begin{equation*}
\begin{aligned}
\|L_{n,\lambda,\sX}^{\le D_n}\|_2^2 &\gtrsim \sum_{\substack{\mu D_n < d < D_n \\ w_n\ |\ d}}\frac{1}{\sqrt{d}} \\ &\gtrsim \frac{1}{\sqrt{w_n}}\left(\sqrt{\frac{D_n}{w_n}}-\sqrt{\frac{\mu D_n}{w_n}}\right)\\ &= (1-\sqrt{\mu})\frac{\sqrt{D_n}}{w_n} \\ &= \omega(1),
\end{aligned}
\end{equation*}
completing the proof.
\end{proof}

\begin{proof}[Proof of Theorem \ref{upbnd_wgn}(b).] For sufficiently large $n$, in Lemma \ref{spsthm} we choose the positive integer 
\[ w_n = \lceil \log(1/\lambda_n) \rceil, \]
which is $o(\sqrt{D_n})$. The divergence of $\|L_{n,\lambda,\sX}^{\le D_n}\|_2^2$ follows from the condition \eqref{ineq}, which is implied by the following sufficient condition: for sufficiently large $n$,
\begin{equation}\label{suf-rho-wgn}
 \rho_n < 0.99\frac{1}{\sqrt{e}}\left(\frac{w_n\cdot 2^{w_n}}{(2w_n)!}\right)^{1/(w-1)} \sqrt{\frac{D_n}{n}}\lambda_n^{w_n/(w_n-1)}.
 \end{equation}
Similar to the proof of Lemma \ref{spsthm2}, notice that
$$\frac{1}{\sqrt{e}}\left(\frac{w_n\cdot 2^{w_n}}{(2w_n)!}\right)^{\frac{1}{w-1}} = \Theta(w_n^{-2}) = \Theta(\log^{-2} (1/\lambda_n));\ \ \ \lambda_n^{w_n/(w_n-1)} = \lambda_n\cdot \lambda_n^{1/(\lceil \log(1/\lambda_n) \rceil-1)} = \Theta(\lambda_n),$$
Thus there exists an absolute constant $C$ such that, if
\[ \rho_n < C\sqrt{\frac{D_n}{n}}\lambda_n\log^{-2}(1/\lambda_n), \]
then \eqref{suf-rho-wgn} is satisfied and the divergence of $\|L_{n,\lambda,\sX}^{\le D_n}\|_2^2$ follows from Lemma \ref{spsthm}.
\end{proof}

\section*{Acknowledgments}
\addcontentsline{toc}{section}{Acknowledgments}
We thank Samuel B.\ Hopkins, Philippe Rigollet, and Eliran Subag for helpful discussions.

\addcontentsline{toc}{section}{References}
\bibliographystyle{alpha}
\bibliography{main}

\appendix

\section{Chernoff Bounds}
In this section, we present two Chernoff-type concentration inequalities used in our proofs.
\begin{lemma}[Local Chernoff bound Gaussian inner products]\label{prodest}
Let $u^{(1)},u^{(2)}\in\mathbb{R}^N$ be independent samples from $\sN(0,I_N)$. Then, for any $0 < t \le N/2$,
\begin{equation*}
\Pr\left[|\la u^{(1)},u^{(2)}\ra| \ge t\right] \le 2\exp\left( -\frac{t^2}{4N}\right).
\end{equation*}
\end{lemma}
\begin{proof}
Since by symmetry $\la u^{(1)},u^{(2)}\ra$ and $-\la u^{(1)},u^{(2)}\ra$ have the same distribution, it suffices to bound $\Pr\left[\la u^{(1)},u^{(2)}\ra \ge t\right]$ for $0 < t\le N/2$. By Markov's inequality on the moment generating function, for any $\mu > 0$,
\[ \Pr\left[\la u^{(1)},u^{(2)}\ra \ge t\right] \le \frac{\EE e^{\mu \la u^{(1)},u^{(2)}\ra}}{e^{\mu t}} = e^{-\mu t}(\EE e^{\mu x_1x_2})^N, \] 
where $x_1,x_2$ are independent samples from $\sN(0,1)$. We compute
\begin{equation*}
\EE e^{\mu x_1x_2} = \frac{1}{2\pi}\iint_{\RR^2}\exp\left(-\frac{x^2}{2}+\mu xy -\frac{y^2}{2}\right)dx\,dy
= (1-\mu^2)^{-\frac{1}{2}}.
\end{equation*}
Take $\mu = t/N \in (0,\frac{1}{2}]$. Note that $1-z \ge e^{-3z/2}$ on $z\in (0,\frac{1}{4}]$, and so $1-\mu^2 \ge e^{-3\mu^2/2}$. Hence
\[ \Pr\left[\la u^{(1)},u^{(2)}\ra \ge t\right] \le \exp\left(-\frac{t^2}{N}\right)\cdot \exp\left(\frac{3N\mu^2}{4}\right) = \exp\left(-\frac{t^2}{4N}\right), \]
and the result follows.
\end{proof}

\noindent
The following result may be found in \cite{LM-chi-squared}.
\begin{lemma}[Chernoff bound for $\chi^2$ distribution]\label{cher-chi2}
For all $0 < z < 1$,
$$\frac{1}{k}\log \Pr\left[\chi_k^2 \le zk\right] \le \frac{1}{2}(1-z+\log z).$$
Similarly, for all $z > 1$,
$$\frac{1}{k}\log \Pr\left[\chi_k^2 \ge zk\right] \le \frac{1}{2}(1-z+\log z).$$
\end{lemma}

\begin{corollary}\label{coro-cher}
For all $0 < t \le 1/2$,
\[ \frac{1}{k}\log \Pr \left[|\chi_k^2 -k| \ge kt\right] \le -\frac{t^2}{3}. \]
\end{corollary}
\begin{proof}
It is easy to check that for $t\in (0,1/2]$,
\begin{align*}
    t+\log(1-t) &\le -\frac{t^2}{3} \\ -t+\log(1+t) &\le -\frac{t^2}{3}.
\end{align*}
Therefore, by Lemma \ref{cher-chi2},
\begin{equation*}
\begin{aligned}
\frac{1}{k}\log \Pr \left[|\chi_k^2 -k| \ge kt\right] &= \frac{1}{k}\log \Pr \left[\chi_k^2 \ge (1+t)k\right] +\frac{1}{k}\log \Pr \left[\chi_k^2 \le (1-t)k\right]\\
&\le \frac{1}{2}(-t+\log(1+t))+\frac{1}{2}(t+\log(1-t)) \\
&\le  -\frac{t^2}{3},
\end{aligned}
\end{equation*}
completing the proof.
\end{proof}

\end{document}